\documentclass{amsart}

\usepackage[utf8]{inputenc}
\usepackage[T1]{fontenc}
\usepackage[english,croatian]{babel}

\usepackage{lmodern}

\usepackage{amsmath}
\usepackage{amsfonts}
\usepackage{amssymb}
\usepackage{amsthm}

\usepackage{color}
\usepackage{float}
\usepackage{wrapfig}
\usepackage{graphicx}
\usepackage{url}

\newcommand{\be}[1]{\textbf{\textit{#1}}}

\newcommand{\mbbZ}{\mathbb{Z}}

\newcommand{\mbbC}{\mathbb{C}}

\newcommand{\mbbP}{\mathbb{P}}

\renewcommand{\dim}{\operatorname{dim}}

\newcommand{\Ker}{\operatorname{Ker}}
\renewcommand{\Im}{\operatorname{Im}}
\newcommand{\Hom}{\operatorname{Hom}}

\newcommand{\Diff}{\operatorname{Diff}}

\newcommand{\End}{\operatorname{End}}
\newcommand{\Id}{\operatorname{Id}}
\newcommand{\diag}{\operatorname{diag}}

\newcommand{\Sp}{\operatorname{Sp}}

\newcommand{\iGr}{\operatorname{iGr}}
\newcommand{\Stab}{\operatorname{Stab}}
\newcommand{\Sym}{\operatorname{Sym}}

\newcommand{\mcaO}{\mathcal{O}}

\newcommand{\mcaF}{\mathcal{F}}

\newcommand{\mfrb}{\mathfrak{b}}
\newcommand{\mfrg}{\mathfrak{g}}
\newcommand{\mfrh}{\mathfrak{h}}

\newcommand{\mfrl}{\mathfrak{l}}

\newcommand{\mfrp}{\mathfrak{p}}
\newcommand{\mfrq}{\mathfrak{q}}
\newcommand{\mfrr}{\mathfrak{r}}
\newcommand{\mfru}{\mathfrak{u}}
\newcommand{\mfrgl}{\mathfrak{gl}}

\newcommand{\mfrsp}{\mathfrak{sp}}

\newcommand{\pbar}{\, |\,}

\newtheorem{theorem}{Theorem}
\newtheorem*{theorem*}{Theorem}
\newtheorem{proposition}[theorem]{Proposition}
\newtheorem*{proposition*}{Proposition}
\newtheorem{lemma}[theorem]{Lemma}
\newtheorem*{lemma*}{Lemma}

\newtheorem*{corollary*}{Corollary}
\theoremstyle{remark}\newtheorem{definition}[theorem]{Definition}
\theoremstyle{remark}\newtheorem*{definition*}{Definition}
\theoremstyle{remark}\newtheorem{example}[theorem]{Example}
\theoremstyle{remark}\newtheorem*{example*}{Example}
\theoremstyle{remark}\newtheorem{remark}[theorem]{Remark}
\theoremstyle{remark}\newtheorem*{remark*}{Remark}
\theoremstyle{remark}\newtheorem{conjecture}[theorem]{Conjecture}
\theoremstyle{remark}\newtheorem*{conjecture*}{Conjecture}

\usepackage[all]{xy}

\usepackage{tikz}
\usetikzlibrary{fit,matrix,positioning}
\usepackage{tkz-graph}

\newlength{\defaultboxsize}
\usepackage[boxsize=\defaultboxsize,aligntableaux=center]{ytableau}

\RequirePackage{tikz}
\usetikzlibrary{decorations.markings}

\newenvironment{dynkin}
{\begin{tikzpicture}[baseline={(0,{-0.7*height("$\alpha_{1_1}$")*1pt})}, decoration={markings,mark=at position 0.6 with {\arrow[ultra thick]{<}}}]}
{\end{tikzpicture}}

\newlength{\dynkinstep}
\newlength{\dynkindotradius}
\newlength{\dynkincrosssize}
\newcommand{\dynkinline}[4]{
\draw(\dynkinstep*#1,\dynkinstep*#2) -- (\dynkinstep*#3,\dynkinstep*#4);}

\newcommand{\dynkindoubleline}[4]{
\draw[double,double distance=.7*\dynkindotradius,postaction={decorate}] (\dynkinstep*#1,\dynkinstep*#2) -- (\dynkinstep*#3,\dynkinstep*#4);}

\newcommand{\dynkindot}[2]{
\draw (\dynkinstep*#1,\dynkinstep*#2) [fill=white] circle  (\dynkindotradius);}

\newcommand{\dynkincross}[2]{
\draw[very thick,white] (#1*\dynkinstep-\dynkincrosssize,#2*\dynkinstep-\dynkincrosssize) -- (#1*\dynkinstep+\dynkincrosssize,#2*\dynkinstep+\dynkincrosssize);
\draw[very thick,white] (#1*\dynkinstep-\dynkincrosssize,#2*\dynkinstep+\dynkincrosssize) -- (#1*\dynkinstep+\dynkincrosssize,#2*\dynkinstep-\dynkincrosssize);
\draw (#1*\dynkinstep-\dynkincrosssize,#2*\dynkinstep-\dynkincrosssize) -- (#1*\dynkinstep+\dynkincrosssize,#2*\dynkinstep+\dynkincrosssize);
\draw (#1*\dynkinstep-\dynkincrosssize,#2*\dynkinstep+\dynkincrosssize) -- (#1*\dynkinstep+\dynkincrosssize,#2*\dynkinstep-\dynkincrosssize);}

\newcommand{\dynkindots}[4]{
\draw[dotted] (\dynkinstep*#1,\dynkinstep*#2) -- (\dynkinstep*#3,\dynkinstep*#4);}

\newcommand{\dynkinlabel}[4]{
\node [#3] at (\dynkinstep*#1,\dynkinstep*#2) {{\scriptsize #4}};}

\newcommand\halfbox[1]{ \tikz[baseline=(n.base)]{
\node(n)[inner sep=1pt]{$#1$};
\draw[line cap=round](n.north west)--(n.south west)--(n.south east);}\smallskip}

\newcommand\spectralsequence[2]{\setlength{\abovedisplayskip}{1em}
\setlength{\belowdisplayskip}{1em}
\ensuremath{#1} \halfbox{\xymatrix@R=.3em@C=.9em{#2}}}

\newcommand\LS[1]{\overline{#1}}

\makeatletter
\renewcommand*\env@matrix[1][*\c@MaxMatrixCols c]{%
  \hskip -\arraycolsep
  \let\@ifnextchar\new@ifnextchar
  \array{#1}}
\makeatother

\usepackage{enumitem}
\setlist[enumerate,1]{label=(\alph*), ref=(\alph*)}
\setlist[enumerate,2]{label=\roman*., ref=\roman*.}
\setlist[enumerate,3]{label=\Alph*., ref=\Alph*.}
\setlist[enumerate,4]{label=\arabic*., ref=\arabic*.}

\allowdisplaybreaks[1] 

\usepackage{mathtools}
\usepackage{hyperref}

\begin{document}

\selectlanguage{english}

\title{Singular BGG complexes for the symplectic case}

\author{Rafael Mrđen}

\address{Faculty of Civil Engineering, University of Zagreb: Fra Andrije Kačića-Miošića 26, 10 000 Zagreb, Croatia.}

\email{rafaelm@grad.hr}

\thanks{The author acknowledges support from the Croatian Science Foundation grant no. 4176, and the QuantiXLie Center of Excellence grant no. KK.01.1.1.01.0004 funded by the European Regional Development Fund.}




\begin{abstract}
Using the Penrose transform, we construct analogues of the BGG (Bernstein-Gelfand-Gelfand) resolutions in certain singular infinitesimal characters, in the holomorphic geometric setting, over the Lagrangian Grassmannian. We prove the exactness of the constructed complex over the big affine cell.
\end{abstract}

\maketitle


\section{Introduction and preliminaries}

The BGG complexes were introduced in \cite{bernstein1975differential} by Bernstein, Gelfand and Gelfand. For a semisimple Lie algebra $\mfrg$ (complex, finite-dimensional), they constructed for each finite-dimensional irreducible $\mfrg$-module $F$ a resolution consisting of direct sums of Verma modules. This construction was generalized by Lepowsky in \cite{lepowsky1977generalization}, from the Borel case to the case of any parabolic subalgebra $\mfrp$. The highest weights of generalized Verma modules appearing in the resolution correspond to $\mfrp$-dominant elements in the affine Weyl group orbit of the highest weight of $F$. These elements can be parametrized by a certain subset of the Weyl group, which can be organized into a directed graph called the Hasse diagram. The Hasse diagram is independent of $F$, so for fixed $(\mfrg,\mfrp)$ all BGG resolutions in regular infinitesimal character have the same shape.

It is well known that homomorphisms of generalized Verma modules correspond to invariant differential operators acting between sheaves of sections of homogeneous vector bundles over the generalized flag manifold $G/P$. On the geometric side, BGG complexes were studied by Čap, Slovák and Souček in \cite{cap2001bernstein}. They constructed BGG complexes in the more general theory of parabolic geometries, for which our $G/P$ is a special case -- the flat model. In the flat model, their construction yields a locally exact resolution of the constant sheaf over $G/P$ defined by $F$, by direct sums of homogeneous vector bundles and invariant differential operators. In case when the parabolic $\mfrp$ is $|1|$-graded, which is equivalent to $G/P$ having structure of a Hermitian symmetric space, the BGG resolution in trivial infinitesimal character coincides with the holomorphic de Rham complex.

Many important operators live in singular infinitesimal character (e.g. the scalar wave operator on the Minkowski space, Dirac-Weyl operators on conformal manifolds, Dirac-Feuter operators on quaternionic manifolds, etc.), and there are no general constructions of resolutions as above in these cases. Several problems emerge here, one of which is a lack of the so called standard operators. So, in order to make a resolution out of the singular orbit, one must construct many non-standard operators. It turned out that the Penrose transform is a particularly useful tool for the construction of such operators. In \cite{pandzic2016bgg}, Pandžić and Souček constructed singular BGG resolutions over the big affine cell in type A, for all maximal parabolics, i.e., all complex Grassmannians. It is visible there that singular BGG resolutions cover the whole singular orbit, and moreover, they have the same shape as certain regular resolutions in lower rank.

Similar results are obtained in this paper, for type C. Here $G$ is the symplectic group $\Sp(2n,\mbbC)$. There is just one $|1|$-graded parabolic $\mfrp$, and $G/P$ is the Lagrangian Grassmannian. We have two types of singularities: singularity of the first kind, involving only short simple roots, and of the second kind, involving also the long simple root. In the construction, we assume that the infinitesimal character is semi-regular, i.e., orthogonal to only one simple root. In the first kind, the constructed BGG complex covers the whole singular orbit. But in the second kind, the orbit decomposes into two complexes, in agreement with Enright-Shelton's theory \cite{enright1987categories}.

For some results in a higher grading, see e.g. \cite{krump2006singular},  \cite{salac2017k-dirac}, \cite{salac2017resolution}.

This paper presents the material from author's PhD thesis \cite{mrden2017singular}. I am grateful to my advisors Pavle Pandžić and  Vladim\'{i}r Souček for their guidance and ideas. Thanks to Tom\'{a}š Salač for helpful discussions.

\subsection{Parabolic subalgebras}

Let $G$ be a semisimple complex Lie group, connected and simply connected, $\mfrg$ its Lie algebra, $\mfrh$ its fixed Cartan subalgebra, and $\Delta^+(\mfrg,\mfrh)$ fixed set of positive roots. The half sum of all the positive roots will be denoted by $\rho$. For an element $w \in W_\mfrg$ in the Weyl group, denote by $l(w)$ the minimal number of simple reflections required to obtain $w$. Denote also
\[ \Phi_{w}:=\left\{ \alpha \in \Delta^+(\mfrg,\mfrh) \ \colon \ w^{-1} \alpha <0 \right\}. \]
A subset $S \subseteq \Delta^+(\mfrg,\mfrh)$ is said to be \be{saturated} if for any $\alpha, \beta \in S$ such that $\alpha + \beta$ is a root, we have $\alpha + \beta \in S$. A subset $S \subseteq \Delta^+(\mfrg,\mfrh)$ is said to be \be{admissible} if both $S$ and $\Delta^+(\mfrg,\mfrh) \setminus S$ are saturated. For $w, w' \in W_\mfrg$ we write $w \stackrel{\alpha}{\longrightarrow} w'$ if $l(w')=l(w)+1$ and $w'=\sigma_\alpha \circ w$, for some $\alpha \in \Delta^+(\mfrg,\mfrh)$, not necessarily simple. We often write only $w \longrightarrow w'$. This way, $W_\mfrg$ becomes a directed graph. Besides the standard action of $W_\mfrg$ on $\mfrh^\ast$, we also use the affine action: $w \cdot \lambda = w(\lambda+\rho)-\rho$. 

Fix a standard parabolic subalgebra $\mfrp = \mfrl \oplus \mfru$ of $\mfrg$. It will be denoted by crossing the nodes in the Dynkin diagram for $\mfrg$ that are not in the Levi factor $\mfrl$. Denote by $\Delta(\mfru)$ the set of positive roots whose root subspaces lie in the nilpotent radial $\mfru$. We write $\mfru^-$ for the opposite nilpotent radical. The \be{(regular) Hasse diagram} of $\mfrp$ is the full subgraph of $W_\mfrg$ with the following nodes:
\[ W^{\mfrp} :=  \left\{ w \in W_\mfrg \ \colon \ \Phi_{w} \subseteq \Delta(\mfru) \right\}. \]
It consists of all elements in $W_\mfrg$ that map $\mfrg$-dominant weights to $\mfrp$-dominant ones. We will mostly be interested in parabolics with abelian nilpotent radical. These are said to be \be{$|1|$-graded} (and also of \be{Hermitian type}). They are necessarily maximal. For classification, see e.g. \cite[2.1.]{enright2014diagrams}. For finding the graph structure of $W^\mfrp$, we will use \cite[3.2.]{cap2009parabolic}:
\begin{proposition}
Suppose $\mfrp \subseteq \mfrg$ is a $|1|$-graded parabolic subalgebra. The map $w \mapsto \Phi_w$ is a bijection from $W^\mfrp$ to the set of all admissible subsets of $\Delta(\mfru)$. A subset $S \subseteq \Delta(\mfru)$ is admissible if and only if the following condition holds:
\begin{equation}
\label{equation:admissible_1_graded}
\text{If } \alpha \in \Delta(\mfru) \text{ and } \beta \in \Delta^+(\mfrl,\mfrh) \text{ such that } \alpha + \beta \in S, \text{ then } \alpha \in S.
\end{equation}

Moreover, $w \stackrel{\alpha}{\longrightarrow} w'$ in $W^\mfrp$ if and only if $|\Phi_{w'}|=|\Phi_w|+1$ and $\Phi_{w'}=\Phi_w\cup \{ \alpha \}$.
\label{proposition:hasse_bijection_saturated}
\end{proposition}

For a weight $\lambda \in \mfrh^\ast$ integral and dominant for $\mfrg$, we write $F(\lambda)$ for the finite-dimensional, irreducible representation of $\mfrg$ with highest weight $\lambda$, and $E(\lambda)$ for its dual. If $\lambda$ is $\mfrp$-dominant, we write $F_\mfrp(\lambda)$ for the finite-dimensional, irreducible representation of $\mfrl$ with highest weight $\lambda$, and with $\mfru$ acting by $0$. We write $E_\mfrp(\lambda)$ for its dual. The same notation will be used for the group representations. In a $|1|$-graded case, the one-dimensional center of $\mfrl$ acts by the scalar $\lambda(E)=\frac{2\langle \lambda, \omega\rangle}{\langle \alpha, \alpha\rangle}$, where $E$ is the \be{grading element} (the unique element from the center of $\mfrl$ acting as $1$ on $\mfru$), $\alpha$ is the crossed simple root, and $\omega$ the corresponding fundamental weight. This scalar is called the \be{generalized conformal weight}.

\subsection{Geometric setup}
The Dynkin notation for $\mfrp$ will also denote the corresponding parabolic subgroup $P \subseteq G$, and the (complex) generalized flag manifold $G/P$. For two standard parabolic subgroups $Q \subseteq P$, the \be{relative Hasse diagram} $W_\mfrp^\mfrq$ of the fibration $G/Q \to G/P$ is the Hasse diagram of the parabolic $\mfrl_{\text{ss}} \cap \mfrq$ in $\mfrl_{\text{ss}}$, where $\mfrl_{\text{ss}}$ is the semisimple part of the Levi factor of $\mfrp$.

Given a finite-dimensional holomorphic representation $\pi \colon P \to \End(V)$, we can form the homogeneous holomorphic vector bundle $G \times_P V \to G/P$. Its holomorphic sections correspond to $V$-valued holomorphic functions on open subsets of $G$ that are $P$-equivariant. For $V=E_\mfrp(\lambda)$, this sheaf is denoted by $\mcaO_\mfrp(\lambda)$.

Recall the relative version of the Bott-Borel-Weil Theorem: Let $\tau \colon G/Q \to G/P$ be the obvious fibration, and $\lambda \in \mfrh^\ast$ be a $\mfrg$-integral and $\mfrp$-dominant weight. If $\lambda+\rho$ is $\mfrp$-singular, all the higher direct images $\tau^q_\ast \mcaO_\mfrq(\lambda)$ are $0$. Otherwise, there is a unique $w \in W_\mfrp \subseteq W_\mfrg$, such that $w \cdot \lambda$ is $\mfrp$-dominant (and necessarily $w^{-1} \in W_\mfrp^\mfrq$). Then, $\tau^{l(w)}_\ast \mcaO_\mfrq(\lambda) \cong \mcaO_\mfrp(w \cdot \lambda)$, and all other higher direct images are $0$.

By an \be{invariant differential operator} we will mean a $\mbbC$-linear differential operator $\mcaO_\mfrp(\lambda) \to \mcaO_\mfrp(\mu)$, invariant with respect to the left translation of sections.

\begin{remark}
\label{remark:Petree}
Peetre's theorem states that any local map between the sections (where ``local'' means that the support of a section is not increased) of vector bundles is necessarily a differential operator. See \cite[V.19.]{kolar1993natural}. 
\end{remark}

\begin{remark}
\label{remark:order=gen_conf_weigt}
In the $|1|$-graded situation, the order of a non-zero invariant differential operator is equal to the difference between the generalized conformal weights in the domain and the codomain. Such an operator is unique up to a non-zero scalar.
\end{remark}

Consider the Borel subgroup $B \subseteq P$. If there exists a non-zero invariant differential operator $\mcaO_\mfrb(\lambda) \to \mcaO_\mfrb(\mu)$, then it is unique up to a scalar (see \cite[11]{baston2016penrose}). The direct image of such a map via $G/B \to G/P$ is again an invariant differential operator, called the \be{standard} operator $\mcaO_\mfrp(\lambda) \to \mcaO_\mfrp(\mu)$. It may be zero, and there may exist invariant differential operators which are \be{non-standard}, for $P \neq B$. Standard operators are in principle completely known, but non-standard ones have not yet been classified. Here is the theorem that we want to find analogues of:

\begin{theorem}[Bernstein-Gelfand-Gelfand-Lepowsky, Čap-Slov\'{a}k-Souček]
\label{theorem:regular_BGG_geometric}
For any $\mfrg$-integral and $\mfrg$-dominant weight $\lambda$, there is a locally exact sequence on $G/P$ resolving the constant sheaf defined by $E(\lambda)$, called the \be{(regular) BGG resolution}:
\begin{equation}
\label{equation:regular_BGG}
0 \to E(\lambda) \to \Delta^\bullet(\lambda), \quad \text{where} \quad \Delta^k(\lambda)= \bigoplus_{w \in W^{\mfrp}, \ l(w)=k} \mcaO_\mfrp(w \cdot \lambda).
\end{equation}
The morphisms are the direct sums of the standard operators $\mcaO_\mfrp(w \cdot \lambda) \to \mcaO_\mfrp(w' \cdot \lambda)$ for $w \rightarrow w'$ in $W^\mfrp$, all of which are non-zero. 
\end{theorem}
See \cite{cap2001bernstein} for a proof in the setting of parabolic geometries.

\subsection{Duality}
There is a contravariant correspondence between the sheaves $\mcaO_\mfrp(\lambda)$ and the generalized Verma modules $M_\mfrp(\lambda) = U(\mfrg) \otimes_{U(\mfrp)} F_\mfrp(\lambda)$. See \cite[11]{baston2016penrose}, \cite[appendix of the preprint]{cap2001bernstein} or \cite{jakobsen1985basic}:
\begin{equation*}
\label{equation:inv_diff_op=verma_hom}
\Diff_G(\mcaO_\mfrp(\lambda),\mcaO_\mfrp(\mu)) \cong \Hom_\mfrg(M_\mfrp(\mu),M_\mfrp(\lambda)).
\end{equation*}

\subsection{Algebraic setup}
Recall the decomposition $\mcaO^\mfrp = \bigoplus_{\lambda \in \mfrh^\ast / W_\mfrg} \mcaO^\mfrp_\lambda$ of the parabolic category $\mcaO^\mfrp$, where $\mcaO^\mfrp_\lambda$ denotes the full subcategory of $\mcaO^\mfrp$ consisting of the modules with generalized infinitesimal character $\lambda$. These subcategories are called the \be{(infinitesimal) blocks} (even though they may be decomposable, as we will see later). Any two blocks with regular generalized infinitesimal characters are mutually equivalent (Jantzen-Zuckerman translation functors), so one usually works only with the so called \be{principal block} $\mcaO^\mfrp_\rho$. The Hasse diagram $W^\mfrp$ parametrizes the $\mfrp$-dominant elements of the affine $W_\mfrg$-orbit of a dominant weight. So, $W^\mfrp$ parametrizes both the generalized Verma modules, and the simple modules in $\mcaO^\mfrp_\rho$. For details, see \cite{humphreys2008representations}. One can do similarly in the singular blocks. Take an integral weight $\lambda \in \mfrh^\ast$ such that $\lambda + \rho$ is dominant, and denote by $\Sigma$ the set of the \be{simple singular roots} for $\lambda$:
\[ \Sigma = \left\{ \alpha \in \Pi \colon \langle \lambda + \rho ,\check{\alpha} \rangle =0 \right\}. \]
The subgroup of $W_\mfrg$ generated by $\{ \sigma_\alpha \colon \alpha \in \Sigma \}$, denoted by $W_\Sigma$, is equal to the stabilizer $\{ z \in W_\mfrg \colon z \cdot \lambda = \lambda \}$. So, $\lambda+\rho$ is regular if and only if $\Sigma = \emptyset$. The \be{singular Hasse diagram} attached to the pair $(\mfrp, \Sigma)$ is
\[ W^{\mfrp, \Sigma} := \left\{ w \in W^\mfrp \colon w \sigma_\alpha \in W^\mfrp \text{ and } w < w \sigma_\alpha, \text{ for all } \alpha \in \Sigma \right\} \subseteq W^\mfrp . \]
\begin{proposition}[\cite{boe2005representation}]
\label{proposition:singular_Hasse}
The singular Hasse diagram $W^{\mfrp, \Sigma}$ is precisely the set of unique minimal length representatives of the left cosets $w W_\Sigma$ of $W_\Sigma$ in $W_\mfrg$ that are contained in $W^\mfrp$. Therefore, $W^{\mfrp, \Sigma}$ parametrizes the $\mfrp$-dominant elements of the affine orbit $W_\mfrg \cdot \lambda$.
\end{proposition}

There is a certain equivalence between a singular block and some regular blocks of some other type, called the \be{Enright-Shelton equivalence}. See \cite[5.5]{enright2014diagrams}.

\subsection{The Penrose transform}

A standard reference is the book \cite{baston2016penrose}. Choose standard parabolic subgroups $P, R \subseteq G$. Their intersection $Q=P \cap R$ is also a standard parabolic subgroup. Choose an open subset $X \subseteq G/P$, and define $Y:=\tau^{-1}(X)$ and $Z:=\eta(Y)$. The subsets $Y$ and $Z$ are open submanifolds of $G/Q$ and $G/R$, respectively. We have the \be{double fibration}, and the \be{restricted double fibration}:
\[ \xymatrix@=.7em{ & G/Q \ar[dl]_{\eta} \ar[dr]^{\tau}& & & & Y \ar[dl]_{\eta} \ar[dr]^{\tau}& \\
G/R & & G/P, & & Z & & X. } \]
The spaces $G/R$ and $Z$ are usually called the \be{twistor spaces}. Start with a weight $\lambda$, $\mfrg$-integral and $\mfrr$-dominant, and form the sheaf $\mcaO_\mfrr(\lambda)$ on $Z \subseteq G/R$. Consider the topological inverse image sheaf $\eta^{-1}\mcaO_{r}(\lambda)$ on $Y$, whose sections correspond to the sections of the pull-back bundle that are constant on the fibers of $\eta$. The weight $\lambda$ remains dominant on the fibers of $\eta$, which themselves are generalized flag manifolds. By resolving $\eta^{-1}\mcaO_{r}(\lambda)|_{\eta^{-1}(x)}$ over each fiber, one obtains an exact sequence of sheaves on $G/Q$ and standard invariant differential operators, called the \be{relative BGG resolution}:
\begin{equation}\label{equation:penrose_relativeBGG}
0 \to \eta^{-1} \mcaO_\mfrr(\lambda) \to \Delta_\eta^\bullet(\lambda), \quad \text{where } \ \Delta_\eta^k(\lambda)= \bigoplus_{w \in W_\mfrr^{\mfrq}, \ l(w)=k} \mcaO_\mfrq(w \cdot \lambda).
\end{equation}
For a full treatment of the relative BGG sequences, see \cite{cap2016relativeI}, \cite{cap2015relativeII}. The \be{hypercohomology spectral sequence} applied to the exact sequence (\ref{equation:penrose_relativeBGG}) has the form
\begin{equation}
\label{equation:penrose_hypercohomology}
E_1^{pq} = H^p(Y,\Delta^q_\eta(\lambda)) \ \Longrightarrow \ H^{p+q}(Y,\eta^{-1}\mcaO_\mfrr(\lambda)).
\end{equation}
Consider the higher direct images along $\tau$ of the sequence (\ref{equation:penrose_relativeBGG}). Let us assume that $X \subseteq G/P$ is an open \be{Stein subset}, for example the big affine cell, or an open ball or a polydisc inside the big affine cell. By the Bott-Borel-Weil Theorem, the sheaves $\tau^q_\ast \Delta^k_\eta(\lambda)$ are locally free, and therefore coherent. Cartan's theorem B implies that for each $k \geq 0$ the Leray spectral sequence for $\tau^q_\ast$ collapses, and gives isomorphisms $H^q(Y,\Delta^k_\eta(\lambda)) \cong \Gamma(X,\tau^q_\ast \Delta^k_\eta(\lambda))$, for $k \geq 0$. This settles the left-hand side of (\ref{equation:penrose_hypercohomology}). For the right-hand side, if the fibers of $\eta \colon Y \to Z$ are smoothly contractible, then there are canonical isomorphims $H^r(Y, \eta^{-1}\mcaO_\mfrr(\lambda)) \cong H^r(Z, \mcaO_\mfrr(\lambda))$ for $r \geq 0$, \cite{buchdahl1983on}.
\begin{theorem}[Baston-Eastwood]
If $X \subseteq G/P$ is Stein, and the fibers of $\eta \colon Y \to Z$ are smoothly contractible, there is a first quadrant spectral sequence:
\begin{equation}\label{equation:penrose_spectral}
E_1^{pq} = \Gamma(X,\tau^q_\ast \Delta^p_\eta(\lambda)) \ \Longrightarrow \ H^{p+q}(Z, \mcaO_\mfrr(\lambda)).
\end{equation}
On the first page, the differentials are standard operators (induced from the relative BGG), but on the other pages we have non-standard invariant differential operators.
\end{theorem}

\section{Structure of the Hasse diagrams}

\subsection{Type C}
We specialize to $G=\Sp(2n,\mbbC)=\begin{dynkin}
\setlength{\dynkinstep}{1cm}
\dynkinline{1}{0}{2}{0}
\dynkindots{2}{0}{3}{0}
\dynkinline{3}{0}{4}{0}
\dynkindoubleline{4}{0}{5}{0}
\dynkindot{1}{0}
\dynkinlabel{1}{0}{below}{$\alpha_1$}
\dynkindot{2}{0}
\dynkinlabel{2}{0}{below}{$\alpha_2$}
\dynkindot{3}{0}
\dynkinlabel{3}{0}{below}{$\alpha_{n-2}$}
\dynkindot{4}{0}
\dynkinlabel{4}{0}{below}{$\alpha_{n-1}$}
\dynkindot{5}{0}
\dynkinlabel{5}{0}{below}{$\alpha_n$}
\end{dynkin}$, the complex symplectic group. Choose the Cartan subalgebra consisting of diagonal matrices $\mfrh \subseteq \mfrg=\mfrsp(2n,\mbbC)$. The positive roots are:
\[ \Delta^+(\mfrg,\mfrh) = \{  a_{ij} = \epsilon_i - \epsilon_j , \quad b_i = 2 \epsilon_i , \quad c_{ij} =  \epsilon_i+ \epsilon_j \ \colon \ 1\leq i < j \leq n \},
\]
where $\epsilon_i$ denotes the projection to $i$-th coordinate. The simple roots are $\alpha_i=a_{i,i+1}$ for $i=1,\ldots,n-1$, and $\alpha_n=b_n$. A weight $\lambda = [\lambda_1, \lambda_2, \ldots, \lambda_n ] \in \mfrh^\ast$ is integral if all $\lambda_i \in \mbbZ$, and dominant if $\lambda_1 \geq \lambda_2 \geq \ldots \lambda_n \geq 0$. A weight is regular if and only if it does not have two coordinates with the same absolute value, and all the coordinates are non-zero. The half sum of all positive roots is $\rho = [n,n-1,\ldots,1]$. The fundamental weights are $\omega_i = \epsilon_1 + \ldots + \epsilon_i$, $1\leq i \leq n$. The Weyl group acts by permutations and sign changes of the coordinates.

Weights for the Levi subalgebra of a standard parabolic subalgebra can be written as $n$-tuples again, but for every crossed node $\alpha_i$ in the Dynkin diagram for the parabolic subalgebra, we will put a bar after the $i$-th coordinate of the weight.

\subsection{|1|-graded parabolic subalgebra}
Fix $\mfrp = \mfrl \oplus \mfru = \begin{dynkin}
\dynkinline{1}{0}{2}{0}
\dynkindots{2}{0}{3}{0}
\dynkinline{3}{0}{4}{0}
\dynkindoubleline{4}{0}{5}{0}
\dynkindot{1}{0}
\dynkindot{2}{0}
\dynkindot{3}{0}
\dynkindot{4}{0}
\dynkincross{5}{0}
\end{dynkin}$, which has $\Delta^+(\mfrl,\mfrh) = \{a_{ij} \colon i<j \}$, and $\Delta(\mfru) = \{b_i  \} \cup \{c_{ij} \colon i<j \}$. Moreover, $\mfrl \cong \mfrgl(n,\mbbC)$, and $\mfru$ consists of the matrices of the form $\begin{pmatrix} 0 & C \\ 0 & 0\end{pmatrix}$, where $C$ is an $n \times n$ symmetric matrix. The grading element is $E=\frac{1}{2}\diag (\underbrace{1,\ldots,1}_n, -1,\ldots,-1)$. The generalized flag manifold corresponding to this parabolic subalgebra is known as the (complex) \be{Lagrangian Grassmannian}, denoted by $\iGr(n,2n)$. It can be realized as the space of all maximal isotropic (Lagrangian) subspaces in a fixed $2n$-dimensional symplectic vector space.

\subsection{Generalized Young diagrams}

Elements of the Hasse diagram $W^\mfrp$ will be represented, using the bijection from Proposition \ref{proposition:hasse_bijection_saturated}, as the admissible subsets of $\Delta(\mfru)$. The point of this identification is that the same proposition also provides a simple criterion for the arrow relation. Note that we can write $\Delta(\mfru)$ in the form that makes the additive structure transparent: Figure \ref{figure:Delta_u_1gr}.

\begin{figure}[h]
$\xymatrix@=1em{
c_{1n}  \ar@{..>}[r]^{+\alpha_{n-1}} & c_{1,n-1} \ar@{..>}[r]^{+\alpha_{n-2}} &  \ldots \ar@{..>}[r]^{+\alpha_2} & c_{12} \ar@{..>}[r]^{+\alpha_1} & b_1 \\
c_{2n} \ar@{..>}[r]\ar@{..>}[u]^{+\alpha_1} & c_{2,n-1} \ar@{..>}[r]\ar@{..>}[u] & \ldots \ar@{..>}[r] & b_2 \ar@{..>}[u] \\
\vdots\ar@{..>}[u]^{+\alpha_2} & \vdots\ar@{..>}[u] & \reflectbox{$\ddots$} \\
c_{n-1,n} \ar@{..>}[r]\ar@{..>}[u]^{+\alpha_{n-2}} & b_{n-1} \ar@{..>}[u] \\
b_n=\alpha_n \ar@{..>}[u]^{+\alpha_{n-1}}}$
\caption[$\Delta(\mfru)$ for $|1|$-graded case.]{$\Delta(\mfru)$ for $\mfrp = \begin{dynkin}
\dynkinline{1}{0}{2}{0}
\dynkindots{2}{0}{3}{0}
\dynkinline{3}{0}{4}{0}
\dynkindoubleline{4}{0}{5}{0}
\dynkindot{1}{0}
\dynkindot{2}{0}
\dynkindot{3}{0}
\dynkindot{4}{0}
\dynkincross{5}{0}
\end{dynkin}$}
\label{figure:Delta_u_1gr}
\end{figure}
\begin{proposition}
Denote $b_i = c_{ii}$. A subset $S \subseteq \Delta(\mfru)$ is admissible if and only if
\begin{equation}
\label{equation:admissible_C}
c_{ij} \in S, \ i \leq j, \quad \Longrightarrow \quad c_{kl} \in S \quad \text{for all} \quad k \geq i, \ l \geq j, \ k \leq l \leq n.
\end{equation}
\end{proposition}
\begin{proof}
Note that sums of the labels of the consecutive arrows in Figure \ref{figure:Delta_u_1gr} are elements of $\Delta^+(\mfrl,\mfrh)$. Then the condition (\ref{equation:admissible_C}) is equivalent to (\ref{equation:admissible_1_graded}).
\end{proof}

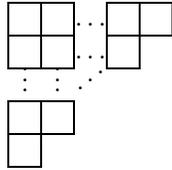
\begin{wrapfigure}{l}{0cm}
\ytableausetup{boxsize=12pt}
\begin{ytableau}
\phantom{c_{1n}} & \phantom{c_{1,n-1}} & \none[\dots]
& \phantom{c_{12}} & \phantom{b_1} \\
\phantom{c_{2n}} & \phantom{c_{2,n-1}} & \none[\dots]
& \phantom{b_2} \\
\none[\vdots] & \none[\vdots]
& \none[\reflectbox{$\ddots$}] \\
\phantom{c_{n-1,n}} & \phantom{b_{n-1}} \\
\phantom{b_n}
\end{ytableau}\ytableausetup{boxsize=\defaultboxsize}
\end{wrapfigure}
An admissible subset $S$ will be represented in the following way: for each element $c_{ij} \in S$, we put a box $\ \ydiagram{1} \ $ on the position $c_{ij}$ in Figure \ref{figure:Delta_u_1gr}. The diagram obtained this way will be called the \be{generalized Young diagram} of the corresponding Hasse diagram element. The maximal admissible subset is $\Delta(\mfru)$ itself, and we denote it by the figure on the left. The condition (\ref{equation:admissible_C}) translated into the generalized Young diagram setting is: for each box in $S$, all the possible boxes bellow, and left of it are again contained in $S$. The notion of the length and the arrow relation transfer very nicely to the generalized Young diagram setting. Namely, the length of an element in $W^\mfrp$ is equal to the number of boxes in the generalized Young diagram. Furthermore, an arrow between elements in $W^\mfrp$ corresponds to the \be{``adding one box'' operation} on the generalized Young diagrams, and the label of that arrow is the same as the label of the added box (follows from Proposition \ref{proposition:hasse_bijection_saturated}). See Figure \ref{figure:Hasse_1gr_n=4}.

\subsection{Lascoux-Sch\"utzenberger (LS) notation}
The idea is from \cite{enright2014diagrams}. Note that a generalized Young diagram is completely determined by a zig-zag line from the top left point to the diagonal. For each move to the right, we write $1$, and for each move down, we write $0$. This way we get a binary sequence of length $n$, written with overline, called the \be{LS notation} of the generalized Young diagram.

\begin{figure}[h]
\begin{gather*}
\scalebox{.8}{\begin{tabular}{cc}
\xymatrix@R=.7em{
\emptyset \ar[r]^{b_4} & \ydiagram{1} \ar[d]_{c_{34}} \\
& \ydiagram{1,1} \ar[r]^{b_3} \ar[d]_{c_{24}} & \ydiagram{2,1} \ar[d] \\
& \ydiagram{1,1,1} \ar[r] \ar[dl]_{c_{14}} & \ydiagram{1,2,1} \ar[d]^{c_{23}} \ar[dl] \\
\ydiagram{1,1,1,1} \ar[r] & \ydiagram{1,1,2,1} \ar[d] & \ydiagram{2,2,1} \ar[r]^{b_2} \ar[dl] & \ydiagram{3,2,1} \ar[dl] \\
& \ydiagram{1,2,2,1} \ar[r] \ar[d] & \ydiagram{1,3,2,1} \ar[d]^{c_{13}} \\
& \ydiagram{2,2,2,1} \ar[r] & \ydiagram{2,3,2,1} \ar[d]^{c_{12}} \\
& & \ydiagram{3,3,2,1} \ar[r]^{b_1} & \ydiagram{4,3,2,1} } & \xymatrix@R=1.7em{
\LS{0000} \ar[r] & \LS{0001} \ar[d] \\
& \LS{0010} \ar[r] \ar[d] & \LS{0011} \ar[d] \\
& \LS{0100} \ar[r] \ar[dl] & \LS{0101} \ar[d] \ar[dl] \\
\LS{1000} \ar[r] & \LS{1001} \ar[d] & \LS{0110} \ar[r] \ar[dl] & \LS{0111} \ar[dl] \\
& \LS{1010} \ar[r] \ar[d] & \LS{1011} \ar[d] \\
& \LS{1100} \ar[r] & \LS{1101} \ar[d] \\
& & \LS{1110} \ar[r] & \LS{1111} }
\end{tabular}} \\
\scalebox{.8}{\xymatrix@=.7em{
[4,3,2,1 \pbar] \ar[r] & [4,3,2,-1 \pbar] \ar[d] \\
& [4,3,1,-2 \pbar] \ar[r] \ar[d] & [4,3,-1,-2 \pbar] \ar[d] \\
& [4,2,1,-3 \pbar] \ar[r] \ar[dl] & [4,2,-1,-3 \pbar] \ar[d] \ar[dl] \\
[3,2,1,-4 \pbar] \ar[r] & [3,2,-1,-4 \pbar] \ar[d] & [4,1,-2,-3 \pbar] \ar[r] \ar[dl] & [4,-1,-2,-3 \pbar] \ar[dl] \\
& [3,1,-2,-4 \pbar] \ar[r] \ar[d] & [3,-1,-2,-4 \pbar] \ar[d] \\
& [2,1,-3,-4 \pbar] \ar[r] & [2,-1,-3,-4 \pbar] \ar[d] \\
& & [1,-2,-3,-4 \pbar] \ar[r] & [-1,-2,-3,-4 \pbar] }}
\end{gather*}
\caption[$W^\mfrp$ and $W^\mfrp \rho$ for $|1|$-graded case, $n=4$]{$W^\mfrp$ and $W^\mfrp \rho$ for $\mfrp = \begin{dynkin}
\dynkinline{1}{0}{2}{0}
\dynkinline{2}{0}{3}{0}
\dynkindoubleline{3}{0}{4}{0}
\dynkindot{1}{0}
\dynkindot{2}{0}
\dynkindot{3}{0}
\dynkincross{4}{0}
\end{dynkin}$}
\label{figure:Hasse_1gr_n=4}
\end{figure}

\begin{proposition}
As directed graphs,
\begin{equation}
\label{equation:LS_objects}
W^\mfrp \cong \left\{\LS{d_1d_2\ldots d_n} \ \colon \ d_i = 0 \text{ or } 1 \right\},
\end{equation}
with the following arrows on the right-hand side:
\begin{equation*}
\label{equation:LS_arrows}
\LS{d_1\ldots 01 \ldots d_n} \to \LS{d_1\ldots 10 \ldots d_n} \quad \text{ and } \quad \LS{d_1\ldots d_{n-1}0} \to \LS{d_1\ldots d_{n-1}1}.
\end{equation*}
Moreover, if $w$ has the digit $1$ on the positions $i_1 < i_2 < \ldots <i_k$, then
\begin{equation}
\label{equation:LS_length}
l(w) = (n+1) \cdot k - \sum_{j=1}^k i_j.
\end{equation}
\end{proposition}
\begin{proof}
The bijection (\ref{equation:LS_objects}) follows from the definition of the LS notation. Obviously, the ``adding one box'' operation has the effect of switching a pair of consecutive digits $\LS{01}$ to $\LS{10}$ (if the added box is not the last possible in a row), or changing the last digit $\LS{0}$ to $\LS{1}$ (if the added box is the last possible in a row). An easy induction on the rank proves the formula (\ref{equation:LS_length}).
\end{proof}
\begin{proposition}
Let $\lambda = [\lambda_1, \lambda_2, \ldots, \lambda_n ] \in \mfrh^\ast$ and $w \in W^\mfrp$. Let $i_1 < i_2 < \ldots <i_k$ denote the positions of the digit $1$ in the LS notation for $w$. Then
\begin{equation}
\label{equation:action_in_LS}
w \lambda = [\lambda_1, \lambda_2, \ldots, \widehat{\lambda_{i_1}}, \ldots, \widehat{\lambda_{i_2}}, \ldots, \widehat{\lambda_{i_k}}, \ldots, \lambda_n, -\lambda_{i_k}, -\lambda_{i_{k-1}}, \ldots, -\lambda_{i_1} \pbar ].
\end{equation}
Coordinates with $\widehat{hat}$ are omitted. In other words, the positions of the digit $1$ are precisely the positions of the coordinates of $\rho$ that become negative in $w \rho$.
\end{proposition}
\begin{proof}
Suppose first that $k=1$. Then $w=\LS{0\ldots010\ldots0}$, and the corresponding generalized Young diagram has just one column of $n+1-i_1$ boxes. It follows that
\[ w = \sigma_{c_{i_1,n}} \circ \ldots \circ \sigma_{c_{n-1,n}} \circ \sigma_{b_n}. \]
Applying this composition to $\lambda$ gives $[\lambda_1, \ldots, \widehat{\lambda_{i_1}}, \ldots, \lambda_n, - \lambda_{i_1}]$.

In general, the same principle is applicable. One can decompose $w$ into columns, and calculate the action of each column from the left to the right. More precisely,
\[ w = C_k \circ C_{k-1} \circ \ldots \circ C_1, \ \text{where } \ C_j = \sigma{c_{i_j+1-j,n+1-j}} \circ \ldots \circ \sigma{c_{n-j,n+1-j}} \circ \sigma{b_{n+1-j}}. \]
Applying this to $\lambda$ gives (\ref{equation:action_in_LS}).
\end{proof}

\subsection{Inductive structure of the regular Hasse diagram}
Denote by $W^n$ our regular Hasse diagram in the rank $n$. The vertices of $W^n$ can be divided into two disjoint sets: $W^n = W^n_0 \sqcup W^n_1$, where $W^n_d$ consists of those LS-words having the first digit $d$. Obviously, each $W^n_d \cong W^{n-1}$ and analogously decomposes further as $W^n_d = W^n_{d0} \sqcup W^n_{d1}$, where each $W^n_{de} \cong W^{n-2}$ as a directed graph. The only arrows between $W^n_0$ and $W^n_1$ are the following: $\LS{01\ldots} \to \LS{10\ldots}$, between $W^n_{01}$ and $W^n_{10}$.

In conclusion, $W^n$ consists of the two pieces $W^{n-1}$ that ``glue'' from the second copy of $W^{n-2}$ in the first piece to the first copy of $W^{n-2}$ in the second piece. See Figure \ref{figure:Hasse_1gr_n=4}.

\subsection{Description of the singular Hasse diagrams}

Take an integral weight $\lambda$ such that $\lambda + \rho$ is dominant and denote by $\Sigma$ its set of simple singular roots. Consider the $W^\mfrp$-orbit of $\lambda + \rho$ (which is the same as the affine $W^\mfrp$-orbit of $\lambda$, up to the shift of coordinates by $\rho$), and look for the elements that are strictly $\mfrp$-dominant. The results that are not strictly $\mfrp$-dominant do not correspond to a homogeneous vector bundle over $G/P$. The remaining part is what is called the \be{singular orbit} attached to the pair $(\mfrp,\Sigma)$ or $(\mfrp,\lambda)$. We can assume that $\Sigma$ does not contain two adjacent simple roots, because otherwise the corresponding block is empty. In the identification (\ref{equation:LS_objects}), we can recognize those LS words that belong to the singular Hasse diagram:
\begin{proposition}
\label{proposition:singular_orbit_via_LS}
Suppose $\Sigma = \{ \alpha_{i_1}, \ldots, \alpha_{i_s} \}$. Denote by $\alpha_n$ the long simple root. If $\alpha_n \not\in \Sigma$, then $ W^{\mfrp,\Sigma} = \left\{\LS{d_1d_2\ldots d_n} \ \colon \ \LS{d_{i_k} d_{i_k +1}} = \LS{01} \ \text{ for } k=1,\ldots,s \right\}$. Otherwise $\alpha_n = \alpha_{i_s}$, and $W^{\mfrp,\Sigma} = \left\{\LS{d_1d_2\ldots d_{n-1} 0} \ \colon \ \LS{d_{i_k} d_{i_k +1}} = \LS{01} \ \text{ for } k=1,\ldots,{s-1} \right\}$.
\end{proposition}
\begin{proof}
Assume first that $\alpha_n \not\in \Sigma$. It is easy to see that the coordinates of $\lambda + \rho$ are strictly decreasing, except on the positions $(i_k,i_k+1)$, where they have an equal value depending on $k$, for $k=1, \ldots, s$. The necessary and sufficient condition for $w = \LS{d_1 \ldots d_n} \in W^\mfrp$ to make $\lambda + \rho$ strictly $\mfrp$-dominant is that for each pair of the adjacent coordinates $(i_k,i_k+1)$, exactly one of them becomes negative. By the formula (\ref{equation:action_in_LS}), this is equivalent to $\LS{d_{i_k} d_{i_k +1}} = \LS{01}$ or $\LS{10}$. From Proposition \ref{proposition:singular_Hasse} we know that $W^{\mfrp,\Sigma}$ consists of minimal length representatives of the left cosets $v W_\Sigma \subseteq W^\mfrp$, where $W_\Sigma = \Stab_{W_\mfrg}(\lambda)$ and $v \in W^\mfrp$. So, we must have $\LS{d_{i_k} d_{i_k +1}} = \LS{01}$ for all $k$.

If $\alpha_n = \alpha_{i_s}$, then also in addition to the previous conditions, the last coordinate of $\lambda + \rho$ is $0$. So, both $\LS{d_1d_2\ldots d_{n-1} 0}$ and $\LS{d_1d_2\ldots d_{n-1} 1}$ are in the same left coset of $W_\Sigma$, and for $W^{\mfrp,\Sigma}$ we choose the shorter one, which is $\LS{d_1d_2\ldots d_{n-1} 0}$.
\end{proof}

Two different cases, depending on whether $\alpha_n \not\in \Sigma$ or $\alpha_n \in \Sigma$, will be referred to as the \be{singularity of the first kind}, and the \be{singularity of the second kind}, respectively. The construction of non-standard operators will be more complicated for the singularity of the second kind. In Figure \ref{figure:singular_orbit_r=4_3}, we give an example of a singular orbit of each kind in rank $4$.
\begin{figure}[h]
\[ \scalebox{.74}{\xymatrix@=.6em{
\times \ar@{.}[r] & [3,2,1,-1 \pbar] \ar@{=}[d] \\
& [3,2,1,-1 \pbar] \ar@{.}[r] \ar@{.}[d] & \times  \ar@{.}[d] \\
& \times \ar@{.}[r] \ar@{.}[dl] & [3,1,-1,-2 \pbar] \ar@{=}[d] \ar[dl] \\
\times \ar@{.}[r] & [2,1,-1,-3 \pbar] \ar@{=}[d] & [3,1,-1,-2 \pbar] \ar[dl] \ar@{.}[r] & \times \ar@{.}[dl] \\
& [2,1,-1,-3 \pbar] \ar@{.}[r] \ar@{.}[d]& \times \ar@{.}[d] \\
& \times \ar@{.}[r] & [1,-1,-2,-3 \pbar] \ar@{=}[d] \\
& & [1,-1,-2,-3 \pbar] \ar@{.}[r] & \times } \ \xymatrix@=.6em{
[3,2,1,0 \pbar] \ar@{=}[r] & [3,2,1,0 \pbar] \ar@{.}[d] \\
& [3,2,0,-1 \pbar] \ar@{=}[r] \ar[d] & [3,2,0,-1 \pbar] \ar[d] \\
& [3,1,0,-2 \pbar] \ar@{=}[r] \ar[dl] & [3,1,0,-2 \pbar] \ar[dl] \ar@{.}[d] \\
[2,1,0,-3 \pbar] \ar@{=}[r] & [2,1,0,-3 \pbar] \ar@{.}[d] & [3,0,-1,-2 \pbar] \ar@{=}[r] \ar[dl] & [3,0,-1,-2 \pbar] \ar[dl] \\
& [2,0,-1,-3 \pbar] \ar@{=}[r] \ar[d] & [2,0,-1,-3 \pbar] \ar[d] \\
& [1,0,-2,-3 \pbar] \ar@{=}[r] & [1,0,-2,-3 \pbar] \ar@{.}[d] \\
& & [0,-1,-2,-3 \pbar] \ar@{=}[r] & [0,-1,-2,-3 \pbar] }} \]
\caption{Singular orbit for $[3,2,1,1]$ and $[3,2,1,0]$} \label{figure:singular_orbit_r=4_3}
\end{figure}
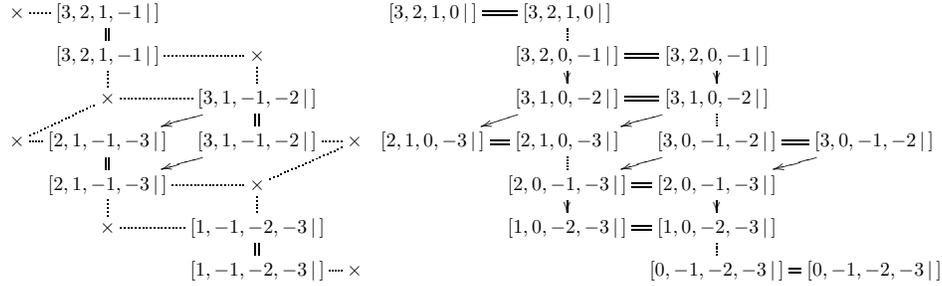
In these orbits there are some non-standard operators, which are not visible. These missing operators will be constructed using the Penrose transform from an appropriately chosen twistor space. Moreover, they will be (together with the standard operators) the differentials in the singular BGG complex.

\section{Construction of non-standard operators}

From now on, we will work with a weight $\lambda$ such that $\lambda + \rho$ is orthogonal to only one simple root. In that case, we say that $\lambda + \rho$ is \be{semi-regular}. So, $\Sigma = \{\alpha_k\}$ for some $k \leq n$. Minimal such $\lambda + \rho$ is
\begin{equation}
\label{equation:lambda_first_kind}
\lambda + \rho = [n-1,n-2, \ldots, n-k+1, \underbrace{\mathbf{n-k}}_{k},\underbrace{\mathbf{n-k}}_{k+1}, n-k-1, \ldots ,2,1]
\end{equation}
for the singularity of the first kind ($k<n$), or $\lambda + \rho = [n-1,n-2, \ldots, 2,1,\mathbf{0}]$ for the singularity of the second kind ($k=n$). We will work with this minimal $\lambda+\rho$, but we want to note that in the construction of the non-standard operators that follows, the minimality is not important, only the order among the coordinates of $\lambda+\rho$ plays a role. Equivalently, one can apply the Jantzen-Zuckerman translation functors to obtain the non-minimal cases. Of course, for a non-minimal $\lambda + \rho$, the orders of the constructed operators will increase (see Remark \ref{remark:order=gen_conf_weigt}).

\subsection{Double fibration}

Fix $k \in \{1,\ldots,n\}$, and form the following double fibration:
\begin{equation}
\label{equation:double_fibration_1_1n_n}
\xymatrix@C=-4em@R=1em{ & G/Q = \begin{dynkin}
\dynkinline{1}{0}{2}{0}
\dynkindots{2}{0}{3}{0}
\dynkinline{3}{0}{4}{0}
\dynkindoubleline{4}{0}{5}{0}
\dynkincross{1}{0}
\dynkindot{2}{0}
\dynkindot{3}{0}
\dynkindot{4}{0}
\dynkincross{5}{0}
\end{dynkin} \ar[dl]_{\eta} \ar[dr]^{\tau}& \\
G/R = \begin{dynkin}
\dynkinline{1}{0}{2}{0}
\dynkindots{2}{0}{3}{0}
\dynkinline{3}{0}{4}{0}
\dynkindoubleline{4}{0}{5}{0}
\dynkindot{5}{0}
\dynkindot{2}{0}
\dynkindot{3}{0}
\dynkindot{4}{0}
\dynkincross{1}{0}
\end{dynkin} & & G/P = \begin{dynkin}
\dynkinline{1}{0}{2}{0}
\dynkindots{2}{0}{3}{0}
\dynkinline{3}{0}{4}{0}
\dynkindoubleline{4}{0}{5}{0}
\dynkindot{1}{0}
\dynkindot{2}{0}
\dynkindot{3}{0}
\dynkindot{4}{0}
\dynkincross{5}{0}
\end{dynkin} . }
\end{equation}
We start with the homogeneous sheaf $\mcaO_\mfrr(\tilde{\lambda})$ on $G/R$, where
\begin{equation}
\label{equation:lambda_tilda}
\tilde{\lambda} + \rho = [n-k \pbar n-1, n-2, \ldots, 2,1] .
\end{equation}
The weight $\tilde{\lambda} = [-k \pbar 0,0, \ldots, 0,0]$ is obviously $\mfrr$-dominant, so the sheaf $\mcaO_\mfrr(\tilde{\lambda})$ is indeed well defined. Recall that $G/P$ can be realized as the Lagrangian Grassmannian $\iGr(n,2n)$, $G/R$ as the isotropic Grassmannian $\iGr(1,2n)$ (biholomorphic to $\mbbP^{2n-1}$), and $G/Q$ as the space of isotropic flags of the type $(1,n)$. More precisely, the double fibration (\ref{equation:double_fibration_1_1n_n}) becomes:
\[ \xymatrix@C=-11em@R=1em{ & \{ (L,W) \colon \dim L=1, \ \dim W =n, \ L \leq W \text{ isotropic} \} \ar[dl]_{\eta} \ar[dr]^{\tau}& \\
\{ L \colon \dim L=1 \} \phantom{ \ L \text{ isotropic}} & & \{ W \colon \dim W =n, \ W \text{ isotropic} \}, } \]
where $\eta$ and $\tau$ are the projections to the components. Take $X \subseteq G/P$ to be the big affine cell. It consists of subspaces spanned by the columns of the matrix $\begin{pmatrix} I \\ C \end{pmatrix}$ (w.r.t a fixed symplectic basis), where $I$ is the identity $n \times n$ matrix, and $C$ a symmetric $n \times n$ matrix. The symmetric matrices $\Sym_n(\mbbC) \cong \mfru^-$ give the cannonical affine coordinates on $X$. Put $Y:=\tau^{-1}(X)$ and $Z:=\eta(Y)$. A general element in the fiber $\tau^{-1}(W)$, $W \in X$, is a pair $(L,W)$, where any non-zero vector in $L$ is a linear combination of the columns of $\begin{pmatrix} I \\ C \end{pmatrix}$. The coefficients in this linear combination are uniquely determined by $L$ up to a non-zero scalar, so they define a point in the projective space $\mbbP^{n-1}$. It follows that we have a biholomorphic bijection $\Sym_n(\mbbC) \times \mbbP^{n-1} \cong Y$ given by $(C, y) \mapsto \left( \begin{pmatrix} y \\ C \cdot y \end{pmatrix}, \begin{pmatrix} I \\ C \end{pmatrix} \right)$, where $y=\begin{pmatrix} y_1 \\ \vdots \\ y_n \end{pmatrix}$ are the projective coordinates. Now, the restricted double fibration is:
\begin{equation}
\label{equation:restricted_double_fibration}
\xymatrix@=.7em{ & Y \ar[dl]_{\eta} \ar[dr]^{\tau}& \\
Z & & X ,} \qquad \xymatrix@=.7em{ & (C, y) \ar@{|->}[dl]_-{\eta} \ar@{|->}[dr]^{\tau}& \\
{\begin{pmatrix} y \\ C \cdot y \end{pmatrix}} & & C .}
\end{equation}
\begin{proposition}
\label{proposition:twistor_space_in_coordinates}
We have $Z=\left\{ \begin{pmatrix} y_1 \\ \vdots \\ y_n \\ z_1 \\ \vdots \\ z_n \end{pmatrix} \in \mbbP^{2n-1} \ \colon \ \text{at least one } y_i \neq 0 \ \right\}$.
\end{proposition}
\begin{proof}
The condition in the curly brackets is necessary because $y$ are projective coordinates. For the converse, assume $y_1 =1$, and observe:
\[ \begin{pmatrix}[c|ccc] z_1- \sum_{i=2}^n z_i \cdot y_i & z_2 & \ldots & z_n \\ \hline
z_2 & 0 & \ldots & 0 \\
\vdots & \vdots &  & \vdots \\ 
z_n & 0 & \ldots & 0 \end{pmatrix} \begin{pmatrix} 1 \\ y_2 \\ \vdots \\ y_n \end{pmatrix} = \begin{pmatrix} z_1 \\ z_2 \\ \vdots \\ z_n \end{pmatrix}. \]
The proof is analogous if some other $y_i=1$.
\end{proof}

\begin{proposition}
The fibers of $\eta \colon Y \to Z$ are smoothly contractible.
\end{proposition}
\begin{proof}
Given $\begin{pmatrix} y \\ z\end{pmatrix} \in Z$, the condition $C \cdot y = z$ is given by linear equations in the entries of the matrix $C$. So, the fiber $\eta^{-1}\begin{pmatrix} y \\ z\end{pmatrix} \subseteq \Sym_n(\mbbC) \times y$ is a certain affine subspace of $Y$, and therefore smoothly contractible.
\end{proof}

Suppose $X' \subseteq X$ is a convex open subset, and put $Y':=\tau^{-1}(X')$, $Z':=\eta(Y')$. In this new restricted double fibration, the fibers of $\eta \colon Y' \to Z'$ are equal to the intersection of an affine set (the fibers of $\eta$ in $Y$) and a convex set (a copy of $X'$ in $Y$), and are therefore also smoothly contractible. So, we have a valid setup for the Penrose transform locally, around any point in the Lagrangian Grassmannian.

\subsection{Relative Hasse diagrams}
To calculate the relative BGG resolution (\ref{equation:penrose_relativeBGG}) of the inverse image $\eta^{-1} \mcaO_\mfrr(\tilde{\lambda})$ on $G/Q$, we need the relative Hasse diagram $W_\mfrr^\mfrq$. The fiber of $\eta$ is $ R/Q \cong \begin{dynkin}
\dynkinline{1}{0}{2}{0}
\dynkindots{2}{0}{3}{0}
\dynkindoubleline{3}{0}{4}{0}
\dynkindot{1}{0}
\dynkindot{2}{0}
\dynkindot{3}{0}
\dynkincross{4}{0}
\dynkinlabel{1}{0}{below}{$\alpha_2$}
\dynkinlabel{4}{0}{below}{$\alpha_n$}
\end{dynkin}$. So, $W_\mfrr^\mfrq$ (and so the relative BGG resolution) has the same shape as the regular Hasse diagram in rank $n-1$. As a subset of $W_\mfrg$, $W_\mfrr^\mfrq$ operates on the last $n-1$ coordinates and ignores the first coordinate of a weight. Therefore, $W_\mfrr^\mfrq$ can be identified with the following subgraph of $W^\mfrp$:
\begin{equation}
\label{equation:left_relative_Hasse}
W_\mfrr^\mfrq = \left\{ w \in W^\mfrp \ \colon \ w=\LS{0 \, d_1 d_2\ldots d_{n-1}} \right\}.
\end{equation}

To apply Bott-Borel-Weil Theorem for calculating the higher direct images along $\tau$, it is convenient to understand the relative Hasse diagram $W_\mfrp^\mfrq$. The fiber of $\tau$ is $P/Q \cong \begin{dynkin}
\dynkinline{1}{0}{2}{0}
\dynkindots{2}{0}{3}{0}
\dynkinline{3}{0}{4}{0}
\dynkincross{1}{0}
\dynkindot{2}{0}
\dynkindot{3}{0}
\dynkindot{4}{0}
\dynkinlabel{1}{0}{below}{$\alpha_1$}
\dynkinlabel{4}{0}{below}{$\alpha_{n-1}$}
\end{dynkin} \cong \mbbP^{n-1}$, also $|1|$-graded, so by Proposition \ref{proposition:hasse_bijection_saturated}:
\begin{equation}
\label{equation:Hasse_Pn}
W_\mfrp^\mfrq = \{ \Id \stackrel{a_{12}}{\longrightarrow} \bullet \stackrel{a_{13}}{\longrightarrow} \bullet \stackrel{a_{14}}{\longrightarrow} \ldots \stackrel{a_{1n}}{\longrightarrow} \bullet \} \subseteq W_\mfrg.
\end{equation}

\begin{example}
\label{example:relative_BGG_3211}
Take $\lambda + \rho = [3,2,1,1]$, so $\tilde{\lambda} + \rho = [1 \pbar 3,2,1]$. The relative BGG resolution of the sheaf $\mcaO_\mfrr(\tilde{\lambda})$ on $G/Q$ is obtained by applying (\ref{equation:left_relative_Hasse}) to $\tilde{\lambda}+ \rho$:
\[ \xymatrix@=.75em{
\eta^{-1} \mcaO_\mfrr(\tilde{\lambda}) \ar[r] & [1 \pbar 3,2,1 \pbar ] \ar[r] & [1 \pbar 3,2,-1 \pbar ] \ar[d] \\
&& [1 \pbar 3,1,-2 \pbar ] \ar[r] \ar[d] & [1 \pbar 3,-1,-2 \pbar] \ar[d] \\
&& [1 \pbar 2,1,-3 \pbar ] \ar[r] & [1 \pbar 2,-1,-3 \pbar ] \ar[d] \\
&& & [1 \pbar 1,-2,-3 \pbar ] \ar[r] & [1 \pbar -1,-2,-3 \pbar ] \ar[r] & 0 . } \]
To calculate higher direct images, remove the first bar in each weight in the resolution. If a weight has two coordinates equal, it is $\mfrp$-singular, and so all its higher direct images are $0$. Otherwise, the surviving higher direct image is obtained by arranging the coordinates in the strictly decreasing order, and the degree is equal to the number of the transpositions of adjacent coordinates needed to move the first coordinate to its correct position (follows from (\ref{equation:Hasse_Pn})). We organize this information on the first page of the spectral sequence (\ref{equation:penrose_spectral}), which in this example is the following:
\[ \spectralsequence{}{
0 & [3,2,1,-1 \pbar] & 0 & 0 & 0 & 0 & 0 \\
0 & 0 & 0 & [3,1,-1,-2 \pbar] \ar[r] & [2,1,-1,-3 \pbar] & 0 & 0 \\
0 & 0 & 0 & 0 & 0 & 0 & [1,-1,-2,-3 \pbar] . } \]
Compare this to Figure \ref{figure:singular_orbit_r=4_3}. Two non-standard operators to be constructed in this case are $[3,2,1,-1 \pbar] \to [3,1,-1,-2 \pbar]$ and $[2,1,-1,-3 \pbar] \to [1,-1,-2,-3 \pbar]$. Note that the objects in the spectral sequence are not really homogeneous sheaves, but rather their sections over $X$. We will omit $\Gamma(X,-)$ from the notation, and write only sheaves, or the defining ($\rho$-shifted) weights, or their LS codes. Also note that a standard operator between two adjacent objects in the relative BGG resolution survives the higher direct image and appears in the spectral sequence as a standard operator, only if both these adjacent objects survive in the same degree. This follows from the functoriality of the direct images.
\end{example}

\begin{example}
Take $\lambda + \rho = [3,2,1,0]$, so $\tilde{\lambda} + \rho = [0 \pbar 3,2,1]$. This is the singularity of the second kind. The relative BGG resolution is the same as in Example \ref{example:relative_BGG_3211}, except that instead of $1$ there is $0$ before the first bar.
Note that now every weight survives a higher direct image. This is typical for the singularity of the second kind. The first page of the spectral sequence (\ref{equation:penrose_spectral}) is:
\[ \scalebox{.69}{\spectralsequence{}{
[3,2,1,0 \pbar] & 0 & 0 & 0 & 0 & 0 & 0 \\
0 & [3,2,0,-1 \pbar] \ar[r] & [3,1,0,-2 \pbar] \ar[r] & [2,1,0,-3 \pbar] & 0 & 0 & 0 \\
0 & 0 & 0 & [3,0,-1,-2 \pbar] \ar[r] & [2,0,-1,-3 \pbar] \ar[r] & [1,0,-2,-3 \pbar] & 0 \\
0 & 0 & 0 & 0 & 0 & 0 & [0,-1,-2,-3 \pbar]. }} \]
Compare this to Figure \ref{figure:singular_orbit_r=4_3}. Two non-standard operators will be construted here: $[3,2,1,0 \pbar] \to [3,0,-1,-2 \pbar]$ and $[2,1,0,-3 \pbar] \to [0,-1,-2,-3 \pbar]$. Namely, the Enright-Shelton equivalence says that this orbit should decompose into two disjoint blocks with respect to the parity, each of the shape $\bullet \to \bullet \to \bullet \to \bullet$.
\end{example}

\begin{example}
In Figure \ref{figure:relative_BGGs_rank_5} we give the degrees of the surviving higher direct images in rank $5$ in all semi-regular cases, from $\Sigma=\{ \alpha_1 \}$ to $\Sigma=\{ \alpha_5 \}$, respectively. The non-standard operators to be constructed are presented with dashed arrows.
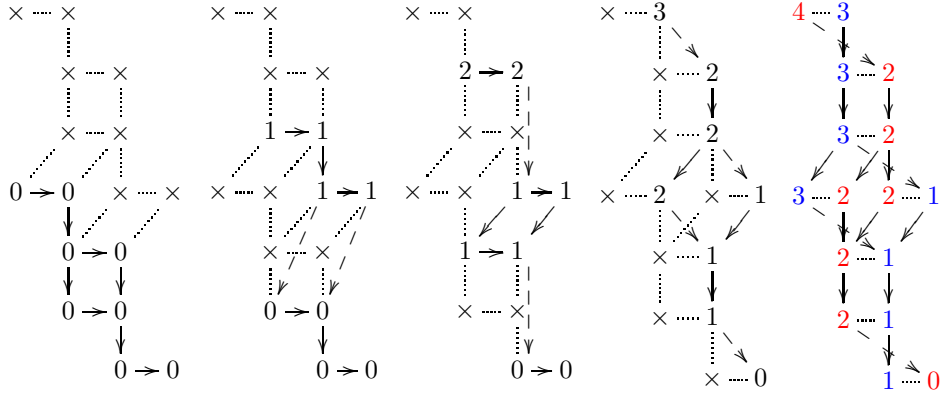
\begin{figure}[h]
\[ \xymatrix@R=1em@C=.6em{
\times \ar@{.}[r] & \times \ar@{.}[d] \\
& \times \ar@{.}[r] \ar@{.}[d] & \times \ar@{.}[d] \\
& \times \ar@{.}[r] \ar@{.}[dl] & \times \ar@{.}[dl] \ar@{.}[d] \\
0 \ar[r] & 0 \ar[d] & \times \ar@{.}[r] \ar@{.}[dl] & \times \ar@{.}[dl] \\
& 0 \ar[r] \ar[d] & 0 \ar[d] \\
& 0 \ar[r] & 0 \ar[d] \\
& & 0 \ar[r] & 0 } \ \xymatrix@R=1em@C=.6em{
\times \ar@{.}[r] & \times \ar@{.}[d] \\
& \times \ar@{.}[r] \ar@{.}[d] & \times \ar@{.}[d] \\
& 1 \ar[r] \ar@{.}[dl] & 1 \ar@{.}[dl] \ar[d] \\
\times \ar@{.}[r] & \times \ar@{.}[d] & 1 \ar[r] \ar@{.}[dl] \ar@{-->}[ldd] & 1 \ar@{.}[dl] \ar@{-->}[ldd] \\
& \times \ar@{.}[r] \ar@{.}[d] & \times \ar@{.}[d] \\
& 0 \ar[r] & 0 \ar[d] \\
& & 0 \ar[r] & 0 } \ \xymatrix@R=1em@C=.6em{
\times \ar@{.}[r] & \times \ar@{.}[d] \\
& 2 \ar[r] \ar@{.}[d] & 2 \ar@{.}[d] \ar@<1ex>@{-->}[dd] \\
& \times \ar@{.}[r] \ar@{.}[dl] & \times \ar@{.}[dl] \ar@{.}[d] \\
\times \ar@{.}[r] & \times \ar@{.}[d] & 1 \ar[r] \ar[dl] & 1 \ar[dl] \\
& 1 \ar[r] \ar@{.}[d] & 1 \ar@{.}[d] \ar@<1ex>@{-->}[dd] \\
& \times \ar@{.}[r] & \times \ar@{.}[d] \\
& & 0 \ar[r] & 0 } \ \xymatrix@R=1em@C=.6em{
\times \ar@{.}[r] & 3 \ar@{.}[d] \ar@{-->}[rd] \\
& \times \ar@{.}[r] \ar@{.}[d] & 2 \ar[d] \\
& \times \ar@{.}[r] \ar@{.}[dl] & 2 \ar[dl] \ar@{.}[d] \ar@{-->}[rd] \\
\times \ar@{.}[r] & 2 \ar@{.}[d] \ar@{-->}[rd] & \times \ar@{.}[r] \ar@{.}[dl] & 1 \ar[dl] \\
& \times \ar@{.}[r] \ar@{.}[d] & 1 \ar[d] \\
& \times \ar@{.}[r] & 1 \ar@{.}[d] \ar@{-->}[rd] \\
& & \times \ar@{.}[r] & 0 } \ \xymatrix@R=1.1em@C=.6em{
\color{red}4 \ar@{.}[r] \ar@{-->}@[red][rdr] & \color{blue}3 \ar@[blue][d]  \\
& \color{blue}3 \ar@{.}[r] \ar@[blue][d] & \color{red}2 \ar@[red][d] \\
& \color{blue}3 \ar@{.}[r] \ar@[blue][dl] \ar@{-->}@[blue][rdr] & \color{red}2 \ar@[red][dl] \ar@[red][d] \\
\color{blue}3 \ar@{.}[r] \ar@{-->}@[blue][rdr] &\color{red} 2 \ar@[red][d] & \color{red}2 \ar@{.}[r] \ar@[red][dl] & \color{blue}1 \ar@[blue][dl] \\
& \color{red}2 \ar@{.}[r] \ar@[red][d] & \color{blue}1 \ar@[blue][d] \\
& \color{red}2 \ar@{.}[r] \ar@{-->}@[red][rdr] & \color{blue}1 \ar@[blue][d] \\
& & \color{blue}1 \ar@{.}[r] & \color{red}0 } \]
\caption{Degrees of higher direct images in rank $5$}
\label{figure:relative_BGGs_rank_5}
\end{figure}
\end{example}

The main technical difference between the two kinds of singularities is the following: In the first kind, all the non-standard operators to be constructed go across zero columns in the first page of the spectral sequence; in the second kind, the wanted non-standard operators go across columns with non-zero entries. We will deal with the two kinds separately.

\subsection{First kind}
Suppose $\lambda+\rho$ is orthogonal to only one short simple root. We work with the minimal such $\lambda+\rho$, given in (\ref{equation:lambda_first_kind}), so $\tilde{\lambda}$ is as in (\ref{equation:lambda_tilda}) for $k<n$.
\begin{proposition}
\label{proposition:relative_BGG_first}
In case of singularity of the first kind, the objects in the relative BGG resolution that survive a higher direct image are parametrized by the LS words of the form $\LS{0 \, d_1 \ldots d_{k-1} \, \mathbf{1} \, d_{k+1} \ldots d_{n-1}}$. The surviving degree is equal to the number of the digits $0$ among $d_1, \ldots ,d_{k-1}$. The result in this degree corresponds to $\LS{d_1 \ldots d_{k-1} \, \mathbf{01} \, d_{k+1} \ldots d_{n-1}} \in W^{\mfrp,\Sigma}$.

Moreover, the first page of the spectral sequence (\ref{equation:penrose_spectral}) agrees with the singular orbit, including both the objects and the standard operators.
\end{proposition}
\begin{proof}
This is just a translation of the Bott-Borel-Weil Theorem in our notation. An element $w =\LS{0 \, d_1 \ldots d_{n-1}} \in W_\mfrr^\mfrq$ will make $\tilde{\lambda}+\rho$ $\mfrp$-regular if and only if it makes the coordinate entry $n-k$ (after the bar) negative. This will happen if and only if $d_k = 1$. The number of the transpositions of adjacent coordinates needed to make $w(\tilde{\lambda}+\rho)$ $\mfrp$-dominant is equal to the number of the coordinates in $w(\tilde{\lambda}+\rho)$ greater than $n-k$ (which can occur only on the positions $2$ to $k$); this equals the number of digits $0$ among $d_1, \ldots, d_{k-1}$. The last two statements are obvious.
\end{proof}

We have an obvious bijection from the singular Hasse diagram (and the surviving part of the relative BGG) to the regular Hasse diagram of rank $n-2$, given by:
\begin{equation}
\label{equation:wannabe_graph_isomorphism}
\LS{d_1 \ldots d_{k-1} \, \mathbf{01} \, d_{k+1} \ldots d_{n-1}} \mapsto \LS{d_1 \ldots d_{k-1} d_{k+1} \ldots d_{n-1}}.
\end{equation}
However, this is not a directed-graph isomorphism. We need to ``add'' more arrows to the left-hand side. Those arrows are the missing non-standard operators, constructed in the following theorem.

\begin{theorem}
\label{theorem:construction_1}
There are non-standard invariant differential operators
\begin{equation}
\label{equation:non_standard_operators_1}
D \colon \mcaO_\mfrp(\nu) \to \mcaO_\mfrp(\nu'')
\end{equation}
for all the pairs $\nu$, $\nu''$ in the singular orbit of the first kind, given by
\begin{equation*}
\nu = \LS{d_1 \ldots d_{k-2} \, 0 \mathbf{01} 1 \, d_{k+3} \ldots d_{n}} \cdot \lambda, \quad \nu'' = \LS{d_1 \ldots d_{k-2} \, 1 \mathbf{01} 0 \, d_{k+3} \ldots d_{n}} \cdot \lambda
\end{equation*}
for $k=2,\ldots,n-2$, or by $\nu = \LS{d_1 \ldots d_{n-3} \, 0 \mathbf{01}} \cdot \lambda$ and $\nu'' = \LS{d_1 \ldots d_{n-3} \, 1 \mathbf{01}} \cdot \lambda$ for $k=n-1$.

If $\lambda$ is minimal as in (\ref{equation:lambda_first_kind}), the operator (\ref{equation:non_standard_operators_1}) is of the order $2$. 
\end{theorem}
\begin{proof}
Take $X'$ to be an open ball inside the big affine cell in $G/P$, and consider the Penrose transform over the corresponding restricted double fibration. In the relative BGG resolution, we find and fix the following sequence:
\begin{equation*}
\xymatrix@=1em{\mu = \LS{0 \, d_1 \ldots d_{k-2} \,0 \mathbf{1} 1\, d_{k+2} \ldots d_{n-1}} \cdot \tilde{\lambda} \ar[d] \\
\mu' = \LS{0 \, d_1 \ldots d_{k-2} \,1 \mathbf{0} 1\, d_{k+2} \ldots d_{n-1}} \cdot \tilde{\lambda} \ar[r] & \mu'' = \LS{0 \, d_1 \ldots d_{k-2} \,1 \mathbf{1} 0\, d_{k+2} \ldots d_{n-1}} \cdot \tilde{\lambda} . }
\end{equation*}
Denote $q=1+$the number of the digits $0$ among $d_1, \ldots, d_{k-2}$. Consider the (part of the) Čech bi-complex that calculates the higher direct images, described in Figure \ref{figure:cech_bi-complex_1}. Here the horizontal morphisms $d_h$ are induced from the differentials of the relative BGG. The vertical morphisms $d_v$ are the usual differentials in the Čech resolution. We have $d_v^2=0$, $d_h^2=0$, and for each square, $d_h d_v =-d_v d_h$. By definition, the vertical cohomologies are equal to the higher direct images of the corresponding sheaves. By Proposition \ref{proposition:relative_BGG_first},
\[ H^q( \check{C}^{\bullet}_{\mu}, d_v ) = \tau_\ast^q \mcaO_\mfrq(\mu) = \mcaO_\mfrp(\nu), \quad H^{q-1}( \check{C}^{\bullet}_{\mu''}, d_v ) = \tau_\ast^{q-1} \mcaO_\mfrq(\mu'') = \mcaO_\mfrp(\nu''). \]
The cochain spaces with nontrivial cohomology are denoted in the bold font. All other vertical cohomologies are trivial, including the complete middle column. 

We will define the operator (\ref{equation:non_standard_operators_1}) on the representatives of the cohomology classes in $H^q( \check{C}^{\bullet}_{\mu}, d_v )$. Take a cocycle $x \in \mathbf{\check{C}^{q}_\mu}$. From $d_v d_h(x) = - d_h d_v(x) =0$ it follows that $d_h(x) \in \check{C}^{q}_{\mu'}$ is a cocycle. Since $H^q( \check{C}^{\bullet}_{\mu'}, d_v ) =0$, it follows that $d_h(x) \in \Im d_v$. So, there is $y \in \check{C}^{q-1}_{\mu'}$ such that $d_v(y) = d_h(x)$. Then, $d_h(y) \in \mathbf{\check{C}^{q-1}_{\mu''}}$, in the correct cochain space. The element $d_h (y)$ is a cocycle: $d_v d_h (y) = - d_h d_v (y ) = - d_h^2 (x) =0$.
\begin{figure}[h]
\[ \xymatrix@=1em{\mathbf{\check{C}^{q}_\mu} \ar[r] & \check{C}^{q}_{\mu'} \ar[r] & \check{C}^{q}_{\mu''} \\
\check{C}^{q-1}_\mu \ar[r] \ar[u] & \check{C}^{q-1}_{\mu'} \ar[r] \ar[u] & \mathbf{\check{C}^{q-1}_{\mu''}} \ar[u] \\
\mcaO_\mfrq(\mu) \ar[r] \ar@{..}[u] & \mcaO_\mfrq(\mu') \ar[r] \ar@{..}[u] & \mcaO_\mfrq(\mu'') \ar@{..}[u]} \quad \xymatrix@=1em{x \in \mathbf{\check{C}^{q}_{\mu}} \ar@{|->}[r] \ar@{|-->}[rrd] & d_h (x) \in \check{C}^{q}_{\mu'} \\
& y \in \check{C}^{q-1}_{\mu'} \ar@{|->}[u] \ar@{|->}[r] & d_h (y) \in \mathbf{\check{C}^{q-1}_{\mu''}} . } \]
\caption{Diagram chasing over the Čech bi-complex (1)}
\label{figure:cech_bi-complex_1}
\end{figure}

Next, we check that we have a well defined map $[x] \mapsto [d_h(y)]$ on the cohomology classes. Take another cocycle $x'$ in the same cohomology class $[x]$, and find $y'$ so that $d_v(y')=d_h(x')$. Since $x-x' = d_v(t)$ for some $t \in \check{C}^{q-1}_\mu$, observe that
\begin{equation}
\label{equation:well_def_on_cocycles}
d_v(y - y' + d_h(t)) = d_h(x-x') - d_h d_v(t) = 0,
\end{equation}
so we conclude that $y - y' + d_h(t) = d_v(t')$ for some $t' \in \check{C}^{q-2}_{\mu'}$. Finally,
\[  d_v (-d_h(t')) = d_h (d_v(t')) =  d_h (y) - d_h(y') + d_h^2(t) = d_h (y) - d_h(y') \in \Im d_v.\]

Therefore, we have a well defined map (\ref{equation:non_standard_operators_1}), given by $D([x])=[d_h(y)]$, which is by construction local, and $G$-invariant. By Remark \ref{remark:Petree}, it is a differential operator. By Remark \ref{remark:order=gen_conf_weigt}, its order is given by the difference of the generalized conformal weights, which is easily seen to be $2$ in the minimal case.
\end{proof}

\begin{definition}
In case of singularity of the first kind, the singular orbit with all the non-standard operators constructed in Theorem \ref{theorem:construction_1} included in it, is called the \be{singular BGG complex} of infinitesimal character $\lambda+\rho$.
\end{definition}

\begin{theorem} In case of singularity of the first kind:
\label{theorem:squares_anticommute_1}
\begin{enumerate}
\item The singular BGG complexes of rank $n$ are directed-graph isomorphic to the regular one of rank $n-2$.
\item
Every square in the singular BGG complex anticommutes.
\item If we add up all objects of the singular BGG complex of the same degree\footnote{The degree in the singular BGG complex is defined using the isomorphism (\ref{equation:wannabe_graph_isomorphism}).}, we get a cochain complex.
\end{enumerate}
\end{theorem}
\begin{proof}
It is easy to check (a): now (\ref{equation:wannabe_graph_isomorphism}) is a directed-graph isomorphism. The statement (c) follows from (b), since $\left(\sum d_i \right)^2 = \sum_{i \neq j} d_i d_j = \sum_{i < j} (d_i d_j + d_j d_i )= 0$.
To prove (b), observe that the standard operators anticommute, since this was already true in the relative BGG resolution. There are no squares with all the operators non-standard. Therefore, we only need to check combination of a standard and a non-standard operator. A typical situation in the relative BGG resolution that induces such a square is ($k$-th coordinate is denoted in the bold font):
\[ \xymatrix@=1em{ \LS{0 \ldots 01 \ldots 0 \mathbf{1}1 \ldots} \ar[r] \ar[d] & \LS{0 \ldots 01 \ldots 1 \mathbf{0}1 \ldots}  \ar[r] \ar[d] & \LS{0 \ldots 01 \ldots 1 \mathbf{1}0 \ldots}  \ar[d] \\
\LS{0\ldots 10 \ldots 0 \mathbf{1}1 \ldots} \ar[r] & \LS{0\ldots 10  \ldots 1 \mathbf{0}1 \ldots} \ar[r] & \LS{0 \ldots 10 \ldots 1 \mathbf{1}0 \ldots}. } \]
(Other possible situations start with $\LS{0 \ldots 0 \mathbf{1}1 \ldots 01 \ldots}$, or $\LS{0\ldots 0 \mathbf{1}1 \ldots 0})$. Denote by $\mu$, $\mu'$, $\mu''$ the objects in the first row, and by $\theta$, $\theta'$, $\theta''$ the objects in the second row, and consider the Čech bi-complex above it, with the same notation as in the proof of Theorem \ref{theorem:construction_1}. Denote by $d$ all the standard operators $\mu \to \theta$, $\mu' \to \theta'$ and $\mu'' \to \theta''$ in the relative BGG resolution. These are just horizontal differentials, but they go in a different direction then those we denoted by $d_h$ in the definition of $D$. The maps $d$ anticommute with both $d_h$ and $d_v$. This is the part of the Čech bi-complex that is mapped to a square in the singular orbit with two parallel standard, and two parallel non-standard operators:
\begin{equation*}
\xymatrix@=.8em{
 & d(x) \in \mathbf{\check{C}^{q}_{\theta}}  \ar@{|->}[r] & d_h ( d(x)) \\
x \in \mathbf{\check{C}^{q}_{\mu}} \ar@{|->}[r] \ar@{|->}[ru] & d_h (x) & y' \ar[u]  \ar@{|->}[u] \ar@{|->}[r] & d_h (y') \in \mathbf{\check{C}^{q-1}_{\theta''}} \\
& y \ar@{|->}[u] \ar@{|->}[r] & d_h (y) \in \mathbf{\check{C}^{q-1}_{\mu''}} \ar@{|-->}[ru]_{-d} .}
\end{equation*}
We need to show that $[d(d_h(y))] = - [d_h(y')]$. First, for $y'$ we can take $d(y)$ without changing the class $D([d(x)])$, since $d_v(d(y))=-d(d_v(y)) = -d(d_h(x)) = d_h(d(x)$). Therefore, $d_h(y') = d_h(d(y)) = - d(d_h(y))$.
\end{proof}

\subsection{Second kind}
Suppose $\lambda+\rho$ is orthogonal only to the long simple root. We work with the minimal such $\lambda+\rho =[n-1,\ldots,1,0 \pbar]$, so $\tilde{\lambda}+\rho=[0 \pbar n-1,\ldots,1]$.
\begin{proposition}
\label{proposition:relative_BGG_second}
In case of singularity of the second kind, all the objects in the relative BGG resolution survive a higher direct image. The surviving degree of an object parametrized by the LS word $w=\LS{0 \, d_1 \ldots d_{n-1}}$ is equal to the number of digits $0$ among $d_1, \ldots, d_{n-1}$. The result of the direct image in this degree corresponds to $\LS{d_1 \ldots d_{n-1} \, 0} \in W^{\mfrp,\Sigma}$. The first page of the spectral sequence (\ref{equation:penrose_spectral}) agrees with the singular orbit, including both the objects and the standard operators.
\end{proposition}
\begin{proof} The same as for Proposition \ref{proposition:relative_BGG_first}.
\end{proof}

We split the singular orbit into the \be{even} and the \be{odd part}, parametrized respectively by the following subsets of the singular Hasse diagram:
\[ W^{\mfrp,\Sigma}_{\epsilon} = \left\{ \LS{d_1 \ldots d_{n-1} \, 0} \text{ with the number of digits $1$ of parity $\epsilon$}\right\}, \ \epsilon \in \{\text{even}, \text{odd}\}. \]
Recall again that the number of the digits $1$ in $w \in W^{\mfrp,\Sigma}$ is equal to the number of the negative coordinates in $w \lambda$. Both $W^{\mfrp,\Sigma}_{\epsilon}$ are in bijection with the regular Hasse diagram of rank $n-2$; in each case, the bijection is:
\begin{equation}
\label{equation:wannabe_graph_isomorphism_2}
\LS{d_1 \ldots d_{n-2} d_{n-1} \, 0} \mapsto \LS{d_1 \ldots d_{n-2}}.
\end{equation}
Again, the idea is to add enough arrows on the left-hand side to make (\ref{equation:wannabe_graph_isomorphism_2}) a directed-graph isomorphism. By inspection, the missing arrows should occur in these situations: $\LS{\ldots 000} \to \LS{\ldots 110}$. For constructing them, we need a crucial fact about the singular orbit of the second kind (see \cite[p. 63]{enright1987categories}):
\begin{lemma}[Enright-Shelton]
\label{lemma:Enright-Shelton_splitting}
There are no non-trivial morphisms between subquotients of objects from the blocks with different parities.
\end{lemma}

\begin{theorem}
\label{theorem:construction_2}
There are non-standard invariant differential operators
\begin{equation}
\label{equation:non_standard_operators_2}
D \colon \mcaO_\mfrp(\nu) \to \mcaO_\mfrp(\nu''')
\end{equation}
for all the pairs $\nu$, $\nu'''$ in the singular orbit of the second kind, given by
\begin{equation*}
\nu = \LS{d_1 \ldots d_{n-3} \, 00 \mathbf{0}} \cdot \lambda, \quad \nu''' = \LS{d_1 \ldots d_{n-3} \, 11 \mathbf{0}} \cdot \lambda.
\end{equation*}
If $\lambda$ is minimal, the operator (\ref{equation:non_standard_operators_2}) is of the order $3$. 
\end{theorem}

\begin{proof}
Take $X'$ to be an open ball inside the big affine cell in $G/P$, and consider the Penrose transform over the corresponding restricted double fibration. In the relative BGG resolution, we find and fix the following sequence:
\begin{equation*}
\xymatrix@=1em{
\mu = \LS{0 \, d_1 \ldots d_{n-3} \, 00} \cdot \tilde{\lambda} \ar[r] & \mu' = \LS{0 \, d_1 \ldots d_{n-3} \, 01} \cdot \tilde{\lambda} \ar[d] \\ & \mu'' = \LS{0 \, d_1 \ldots d_{n-3} \, 10} \cdot \tilde{\lambda} \ar[r] & \mu''' = \LS{0 \, d_1 \ldots d_{n-3} \, 11} \cdot \tilde{\lambda}.}
\end{equation*}
Let $q=2+$the number of digits $0$ in $d_1, \ldots, d_{n-3}$. Let also $\nu' =\LS{d_1 \ldots d_{n-3} \, 01 0} \cdot \lambda$ and $\nu'' = \LS{d_1 \ldots d_{n-3} \, 10 0} \cdot \lambda$. Consider again the Čech bi-complex, described in Figure \ref{figure:chech_bi-complex_2}. By Proposition \ref{proposition:relative_BGG_second}, we have the following:
\begin{align*}
&H^q( \check{C}^{\bullet}_{\mu}, d_v ) = \tau_\ast^q \mcaO_\mfrq(\mu) = \mcaO_\mfrp(\nu), &  &\hskip -.25cm H^{q-1}( \check{C}^{\bullet}_{\mu'}, d_v ) = \tau_\ast^{q-1} \mcaO_\mfrq(\mu') = \mcaO_\mfrp(\nu'), \\
&H^{q-1}( \check{C}^{\bullet}_{\mu''}, d_v ) = \tau_\ast^{q-1} \mcaO_\mfrq(\mu'') = \mcaO_\mfrp(\nu''), & &\hskip -.25cm H^{q-2}( \check{C}^{\bullet}_{\mu'''}, d_v ) = \tau_\ast^{q-2} \mcaO_\mfrq(\mu''') = \mcaO_\mfrp(\nu'''),
\end{align*}
and all other vertical cohomologies are trivial. Note that we also have a standard operator $\mcaO_\mfrq(\mu') \to \mcaO_\mfrq(\mu'')$ that survives on the $(q-1)$-th cohomology,
\[ d \colon H^{q-1}( \check{C}^{\bullet}_{\mu'}, d_v ) \to H^{q-1}( \check{C}^{\bullet}_{\mu''}, d_v ), \qquad d([y]) = [ d_h(y) ]. \]

As before, take a cocycle $x \in \mathbf{\check{C}^{q}_\mu}$, and find $y \in \check{C}^{q-1}_{\mu'}$ such that $d_v(y) = d_h(x)$. Then, $d_h(y) \in \mathbf{\check{C}^{q-1}_{\mu''}}$ is also a cocycle. But since $H^{q-1}( \check{C}^{\bullet}_{\mu''}, d_v ) \neq 0$, we cannot conclude that $d_h(y) \in \Im d_v$ and proceed in the same way. To overcome this, we claim that the map
\begin{equation}
\label{equation:trivial_map}
H^q( \check{C}^{\bullet}_{\mu}, d_v ) \to \raisebox{.2em}{$H^{q-1}( \check{C}^{\bullet}_{\mu''}, d_v )$} \big/ \raisebox{-.2em}{$\Im d$}, \qquad  [x] \mapsto [d_h(y)] + \Im d
\end{equation}
is well-defined. Take $x' = x+ d_v(t)$ and choose $y'$ so that $d_h(x')=d_v(y')$. The equation (\ref{equation:well_def_on_cocycles}) shows that $[y-y'+d_h(t)] \in H^{q-1}( \check{C}^{\bullet}_{\mu'}, d_v )$. Moreover, observe that $d([y-y'+d_h(t)]) = [d_h(y)]-[d_h(y')] \in \Im d$, which proves our claim. Since obviously $\nu$ and $\nu''$ are of different parity, Lemma \ref{lemma:Enright-Shelton_splitting} implies that the map in (\ref{equation:trivial_map}) is trivial. Unwinding, this means that we can find a cocycle $y'' \in \mathbf{\check{C}^{q-1}_{\mu'}}$ so that $d_h(y)-d_h(y'') \in \Im d_v$. Consequently, we can replace $y$ by $y-y''$ and continue our diagram chase downwards, since now:
\begin{align*}
&d_v(y-y'') = d_h(x), \\
&d_h(y-y'') = d_v(z), \text{ for some } z \in \check{C}^{q-2}_{\mu''}.
\end{align*}
Finally, $d_h(z) \in \mathbf{\check{C}^{q-2}_{\mu'''}}$ is a cocycle. We want to define $D([x]) = [d_h(z)]$.
\begin{figure}[h]
\[ \xymatrix@=1em{\mathbf{\check{C}^{q}_\mu} \ar[r] & \check{C}^{q}_{\mu'} \ar[r] & \check{C}^{q}_{\mu''} \ar[r] & \check{C}^{q}_{\mu'''} \\
\check{C}^{q-1}_\mu \ar[r] \ar[u] & \mathbf{\check{C}^{q-1}_{\mu'}} \ar[r] \ar[u] & \mathbf{\check{C}^{q-1}_{\mu''}} \ar[u] \ar[r] & \check{C}^{q-1}_{\mu'''} \ar[u] \\
\check{C}^{q-2}_\mu \ar[r] \ar[u] & \check{C}^{q-2}_{\mu'} \ar[r] \ar[u] & \check{C}^{q-2}_{\mu''} \ar[u] \ar[r] & \mathbf{\check{C}^{q-2}_{\mu'''}} \ar[u] \\
\mcaO_\mfrq(\mu) \ar[r] \ar@{..}[u] & \mcaO_\mfrq(\mu') \ar[r] \ar@{..}[u] & \mcaO_\mfrq(\mu'') \ar[r] \ar@{..}[u] & \mcaO_\mfrq(\mu''') \ar@{..}[u]} \ \xymatrix@=1em{x \ar@{|->}[r] & d_h(x) \\
& y-y'' \ar@{|->}[u] \ar@{|->}[r] & d_h(y-y'') \\
& & z \ar@{|->}[u] \ar@{|->}[r] & d_h(z) . } \]
\caption{Diagram chasing over the Čech bi-complex (2)}
\label{figure:chech_bi-complex_2}
\end{figure}
It still needs to be checked that $D$ does not depend on the various choices we made. For this, we introduce another auxiliary map
\begin{equation}
\tilde{d} \colon \Ker d \to H^{q-2}( \check{C}^{\bullet}_{\mu'''}, d_v ),
\end{equation}
defined as follows. For $[y] \in \Ker d$, we can choose $z \in \check{C}^{q-2}_{\mu''}$ such that $d_v(z) = d_h(y)$. We put $\tilde{d}([y]) := [d_h(z)]$. It is easy to check that $\tilde{d}$ is well defined (in the same way as for $D$ in the proof of Theorem \ref{theorem:construction_1}). Since $\nu'$ and $\nu'''$ are of different parity, Lemma \ref{lemma:Enright-Shelton_splitting} implies that the map $\tilde{d}$ is trivial. Suppose we have $x, x' \in \mathbf{\check{C}^{q}_\mu}$ such that $x-x'=d_v(t)$ for some $t \in \check{C}^{q-1}_\mu$, and consider different choices for defining $D$:
\[ \xymatrix@=1em{x \ar@{|->}[r] & d_h(x) \\
& y \ar@{|->}[u] \ar@{|->}[r] & d_h(y) \in \Im d_v \\
& & z \ar@{|->}[u] \ar@{|->}[r] & d_h(z) , } \xymatrix@=1em{x' \ar@{|->}[r] & d_h(x') \\
& y' \ar@{|->}[u] \ar@{|->}[r] & d_h(y') \in \Im d_v \\
& & z' \ar@{|->}[u] \ar@{|->}[r] & d_h(z') . } \]
The equation (\ref{equation:well_def_on_cocycles}) again shows that $[y-y'+d_h(t)] \in H^{q-1}( \check{C}^{\bullet}_{\mu'}, d_v )$. Observe that $d([y-y'+d_h(t)]) = [d_h(y) - d_h(y')] = [d_v(z-z')] = 0$. So, $[y-y'+d_h(t)] \in \Ker d$, and therefore $0 = \tilde{d}( [y-y'+d_h(t)]) = [d_h(z-z')] = [d_h(z)]-[d_h(z')]$. The last conclusions are analogous to the ones in case of singularity of the first kind.
\end{proof}

\begin{definition}
In case of singularity of the second kind, the even (resp. the odd) part of the singular orbit, with all the non-standard operators constructed in Theorem \ref{theorem:construction_2} included in it, is called the \be{even} (resp. the \be{odd}) \be{singular BGG complex} of infinitesimal character $\lambda+\rho$.
\end{definition}

\begin{theorem} In case of singularity of the second kind:
\label{theorem:squares_anticommute_2}
\begin{enumerate}
\item The singular orbit of rank $n$ consists of two singular BGG complexes, each of which is directed-graph isomorphic to the regular one of rank $n-2$.
\item
Every square in the singular BGG complex anticommutes.
\item
If we add up all objects of the singular BGG complex of the same degree, we get a cochain complex.
\end{enumerate}
\end{theorem}
\begin{proof}
Part (a) is trivial because of (\ref{equation:wannabe_graph_isomorphism_2}), and (c) follows from (b). To prove (b), the situations to consider in the relative BGG resolution are the following:
\[ \xymatrix@=1em{ \LS{0\ldots 01 \ldots 00} \ar[r] \ar[d] & \LS{0\ldots 01 \ldots 01} \ar[r] \ar[d] & \LS{0\ldots 01 \ldots 10}  \ar[r] \ar[d] & \LS{0\ldots 01 \ldots 11} \ar[d] \\
\LS{0\ldots 10 \ldots 00} \ar[r] & \LS{0\ldots 10 \ldots 01} \ar[r]& \LS{0\ldots 10 \ldots 10}  \ar[r] & \LS{0\ldots 10 \ldots 11} . } \]

Denote by $\mu$, $\mu'$, $\mu''$, $\mu'''$ the objects in the first row, by $\theta$, $\theta'$, $\theta''$, $\theta'''$ the objects in the second row, and by $d$ all the standard operators $\mu \to \theta$ and the primed versions. In the Čech bi-complex we have:
\begin{equation*}
\xymatrix@=.8em{
 & d(x) \in \mathbf{\check{C}^{q}_{\theta}}  \ar@{|->}[r] & d_h ( d(x)) \\
x \in \mathbf{\check{C}^{q}_{\mu}} \ar@{|->}[r] \ar@{|->}[ru] & d_h (x) & y' \ar[u]  \ar@{|->}[u] \ar@{|->}[r] & d_h (y') \in \mathbf{\check{C}^{q-1}_{\theta''}} \\
& y \ar@{|->}[u] \ar@{|->}[r] & d_h (y) \in \mathbf{\check{C}^{q-1}_{\mu''}} & z' \ar@{|->}[u] \ar@{|->}[r] & d_h(z') \in \mathbf{\check{C}^{q-2}_{\theta'''}} \\
& & z \ar@{|->}[u] \ar@{|->}[r] & d_h(z) \in \mathbf{\check{C}^{q-2}_{\mu'''}} \ar@{|-->}[ru]_{-d} . }
\end{equation*}
We need to show that $[d(d_h(z))] = - [d_h(z')]$. In the proof of Theorem \ref{theorem:squares_anticommute_1} we saw that we can take $y' = d(y)$, and then $d_h(y') = - d(d_h(y))$. In the same way $d_v(d(z)) = -d(d_v(z))= -d(d_h(y)) = d_h(y')$, so $z'=d(z)$ is a good candidate. Finally, $d_h(z') = d_h(d(z)) = - d(d_h(z))$.
\end{proof}

\section{Exactness of the singular BGG complex}

\begin{lemma}
\label{lemma:vanishing}
Let $Z$ be the twistor space of the restricted double fibration (\ref{equation:restricted_double_fibration}).
For any coherent sheaf $\mcaF$ on $Z$ we have:
\begin{equation}
\label{equation:vanishing}
H^k(Z,\mcaF) = 0, \quad \text{for all } k \geq n.
\end{equation}
\end{lemma}

\begin{proof}
From Proposition \ref{proposition:twistor_space_in_coordinates} it is obvious that $Z$ is a union of $n$ open subsets given by the equations $y_i \neq 0$, for $i=1,\ldots,n$. Each of those is isomorphic to $\mbbC^{2n-1}$, hence affine. Cartan's Theorem B and the Leray Theorem imply the vanishing (\ref{equation:vanishing}).
\end{proof}

\begin{theorem}
\label{theorem:exactness}
Each singular BGG complex is exact (in positive degrees) over the big affine cell $X$.
\end{theorem}
\begin{proof}
Observe that the spectral sequence (\ref{equation:penrose_spectral}) has on the abutment $E_\infty$ the sections over $X$ of the cohomologies of our singular BGG complex. This follows from the construction: non-standard operators were defined exactly as the induced differentials in the hypercohomology spectral sequence, so they appear on the last page of the spectral sequence before it stabilizes. Moreover, by Cartan's Theorem B, the functor $\Gamma(X,-)$ is exact, so it commutes with taking cohomology of a cochain complex of sheaves.

From Proposition \ref{proposition:relative_BGG_first}, Proposition \ref{proposition:relative_BGG_second} and (\ref{equation:LS_length}), it follows that the non-trivial elements on the first page of the spectral sequence with the smallest $p+q$ are respectively $\LS{0 \ldots 0 \mathbf{1} 0 \ldots 0}$ for the first kind, and $\LS{0 \ldots 0}$ and $\LS{0 \ldots 01}$ for the second kind, and each of them has $p+q = n-1$. So $H^i(Z, \mcaO_\mfrr(\tilde{\lambda}))$ measure the non-exactness of the singular BGG complex, up to the shift in degree by $n-1$. Because of Lemma \ref{lemma:vanishing}, singular BGG complexes are exact from the degree $n-(n-1)=1$ above.
\end{proof}

\begin{example}
The even singular BGG complex for $\lambda + \rho = [2,1,0]$
consists of one non-standard operator
$D \colon \mcaO_\mfrp(\begin{dynkin}
\dynkinline{0}{0}{1}{0}
\dynkindoubleline{1}{0}{2}{0}
\dynkindot{0}{0}
\dynkindot{1}{0}
\dynkincross{2}{0}
\dynkinlabel{0}{0}{above}{$0$}
\dynkinlabel{1}{0}{above}{$0$}
\dynkinlabel{2}{0}{above}{$-1$}
\end{dynkin}) \to \mcaO_\mfrp(\begin{dynkin}
\dynkinline{0}{0}{1}{0}
\dynkindoubleline{1}{0}{2}{0}
\dynkindot{0}{0}
\dynkindot{1}{0}
\dynkincross{2}{0}
\dynkinlabel{0}{0}{above}{$0$}
\dynkinlabel{1}{0}{above}{$0$}
\dynkinlabel{2}{0}{above}{$-3$}
\end{dynkin})$, surjective over the big cell. By finding the maximal vector of the corresponding homomorphism, we can find the formula for $D$ in the local coordinates on the big cell given by $\mfru^-$:
\[ D = 4 \partial_{b_1} \partial_{b_2} \partial_{b_3} - \partial_{b_1} \partial_{c_{23}}^2 - \partial_{b_2} \partial_{c_{13}}^2 - \partial_{b_3} \partial_{c_{12}}^2 +  \partial_{c_{12}} \partial_{c_{13}} \partial_{c_{23}}.\]
\end{example}

\subsection{Conjectures}
The following conjecture would imply local exactness of the singular BGG complex, that is, exactness in the category of sheaves:
\begin{conjecture}
\label{conjecture:local_exactness}
The conclusion of Lemma \ref{lemma:vanishing} is true for the twistor space $Z'$, where $X'$ is a suitably chosen, but arbitrarily small open subset.
\end{conjecture}

If $\lambda$ is of a higher singularity, say $|\Sigma|=s>1$, a reasonable thing to try would be the Penrose transform over:
\begin{equation}
\label{equation:double_fibration_s_n}
\xymatrix@C=-4em@R=.4em{ & \begin{dynkin}
\dynkindots{2}{0}{3}{0}
\dynkinline{3}{0}{4}{0}
\dynkinline{4}{0}{5}{0}
\dynkindots{5}{0}{6}{0}
\dynkindoubleline{6}{0}{7}{0}
\dynkindot{2}{0}
\dynkindot{3}{0}
\dynkincross{4}{0}
\dynkindot{5}{0}
\dynkindot{6}{0}
\dynkincross{7}{0}
\dynkinlabel{4}{0}{below}{$\alpha_s$}
\dynkinlabel{7}{0}{below}{$\alpha_n$}
\end{dynkin} \ar[dl] \ar[dr] & \\
\begin{dynkin}
\dynkindots{2}{0}{3}{0}
\dynkinline{3}{0}{4}{0}
\dynkinline{4}{0}{5}{0}
\dynkindots{5}{0}{6}{0}
\dynkindoubleline{6}{0}{7}{0}
\dynkindot{2}{0}
\dynkindot{3}{0}
\dynkincross{4}{0}
\dynkindot{5}{0}
\dynkindot{6}{0}
\dynkindot{7}{0}
\dynkinlabel{4}{0}{below}{$\alpha_s$}
\dynkinlabel{7}{0}{below}{$\alpha_n$}
\end{dynkin} & & \begin{dynkin}
\dynkinline{1}{0}{2}{0}
\dynkindots{2}{0}{3}{0}
\dynkinline{3}{0}{4}{0}
\dynkindoubleline{4}{0}{5}{0}
\dynkindot{1}{0}
\dynkindot{2}{0}
\dynkindot{3}{0}
\dynkindot{4}{0}
\dynkincross{5}{0}
\dynkinlabel{5}{0}{below}{$\alpha_n$}
\end{dynkin} , }
\end{equation}
In this setting, the appropriate vanishing result would be the following:
\begin{conjecture}
\label{conjecture:higher_singularity}
Let $X \subseteq \iGr(n,2n)$ be the big affine cell, or a ball or a polydisc inside it, and let $Z$ be the corresponding twistor space in the double fibration (\ref{equation:double_fibration_s_n}). For any coherent sheaf $\mcaF$ on $Z$ we have:
\[ H^k(Z,\mcaF) = 0, \quad \text{for all } k > s(n-s) - \frac{s(s-1)}{2}. \]
\end{conjecture}

\bibliographystyle{alpha}
\bibliography{My_BibTex_Library_shorter}

\begin{thebibliography}{BGG75}

\bibitem[BE16]{baston2016penrose}
R.~J. Baston and M.~G. Eastwood.
\newblock {\em The Penrose Transform: Its Interaction with Representation
  Theory}.
\newblock Dover Books on Mathematics. Dover Publications, 2016.

\bibitem[BGG75]{bernstein1975differential}
I.~N. Bernstein, I.~M. Gelfand, and S.~I. Gelfand.
\newblock Differential operators on the base affine space and a study of
  {$\mathfrak{g}$}-modules.
\newblock In I.~M. Gelfand, editor, {\em Lie groups and their representations},
  pages 21--64. Adam Hilger, 1975.

\bibitem[BN05]{boe2005representation}
B.~D. Boe and D.~K. Nakano.
\newblock {Representation type of the blocks of category $\mathcal{O}_S$}.
\newblock {\em Advances in Mathematics}, 196:193--256, 2005.

\bibitem[Buc83]{buchdahl1983on}
N.~Buchdahl.
\newblock On the relative de {R}ham sequence.
\newblock {\em Proc.Amer.Math.}, 87(2):363--366, 1983.

\bibitem[{\v{C}}S09]{cap2009parabolic}
A.~{\v{C}}ap and J.~Slov{\'a}k.
\newblock {\em Parabolic Geometries I: Background and General Theory}, volume
  154 of {\em Mathematical Surveys and Monographs}.
\newblock American Mathematical Society, 2009.

\bibitem[{\v{C}}S15]{cap2015relativeII}
A.~{\v{C}}ap and V.~Sou\v{c}ek.
\newblock {Relative BGG sequences II}, 2015.
\newblock arXiv:1510.03986 [math.DG].

\bibitem[{\v{C}}S16]{cap2016relativeI}
A.~{\v{C}}ap and V.~Sou\v{c}ek.
\newblock {Relative BGG sequences I}.
\newblock {\em J. Algebra}, 463:188--210, 2016.

\bibitem[{\v{C}}SS01]{cap2001bernstein}
A.~{\v{C}}ap, J.~Slov\'{a}k, and V.~Sou\v{c}ek.
\newblock {B}ernstein-{G}elfand-{G}elfand sequences.
\newblock {\em Annals of Mathematics}, 154(1):97--113, 2001.
\newblock preprint: arXiv:math/0001164 [math.DG].

\bibitem[EHP14]{enright2014diagrams}
T.~Enright, M.~Hunziker, and A.~Pruett.
\newblock Diagrams of hermitian type, highest weight modules, and syzygies of
  determinantal varieties.
\newblock In Howe, Hunziker, and Willenbring, editors, {\em Symmetry:
  Representation Theory and Its Applications}. Birkh{\"a}user, 2014.

\bibitem[ES87]{enright1987categories}
T.~J. Enright and B.~Shelton.
\newblock {\em {Categories of Highest Weight Modules: Applications to Classical
  Hermitian Symmetric Pairs}}, volume 67 (367) of {\em Memoirs AMS}.
\newblock 1987.

\bibitem[Hum08]{humphreys2008representations}
J.~E. Humphreys.
\newblock {\em Representations of Semisimple Lie Algebras in the BGG Category
  {$\mathcal O$}}, volume~94 of {\em Graduate Studies in Mathematics}.
\newblock American Mathematical Society, 2008.

\bibitem[Jak85]{jakobsen1985basic}
H.~P. Jakobsen.
\newblock {Basic covariant differential operators on Hermitian symmetric
  spaces}.
\newblock {\em {Annales Scientifiques de l'Ecole Normale Superieure}},
  {18}({3}):{421--436}, {1985}.

\bibitem[KMS93]{kolar1993natural}
I.~Kol\'{a}\v{r}, P.~W. Michor, and J.~Slov\'{a}k.
\newblock {\em Natural Operations in Differential Geometry}.
\newblock Springer, 1993.

\bibitem[KS06]{krump2006singular}
L.~Krump and V.~Souček.
\newblock {Singular BGG sequences for the even orthogonal case}.
\newblock {\em Archivum Mathematicum (Brno)}, Tomus 42, Supplement:267--278,
  2006.

\bibitem[Lep77]{lepowsky1977generalization}
J.~Lepowsky.
\newblock A generalization of the {B}ernstein-{G}elfand-{G}elfand resolution.
\newblock {\em Journal of Algebra}, 49(2):496--511, 1977.

\bibitem[Mr{\dj}17]{mrden2017singular}
R.~Mr{\dj}en.
\newblock {\em Singular {BGG} complexes for the symplectic case}.
\newblock PhD thesis, University of Zagreb, 2017.
\newblock Available at \url{http://bib.irb.hr/prikazi-rad?&rad=904058}.

\bibitem[PS16]{pandzic2016bgg}
P.~Pand\v{z}i\'{c} and V.~Sou\v{c}ek.
\newblock {BGG} complexes in singular infinitesimal character for type {A}.
\newblock {\em to appear in J. Math. Phys.}, 2016.
\newblock arXiv:1612.05946 [math.DG].

\bibitem[Sal17a]{salac2017k-dirac}
T.~Sala\v{c}.
\newblock {$k$-Dirac complexes}, 2017.
\newblock arXiv:1705.09469 [math.DG].

\bibitem[Sal17b]{salac2017resolution}
T.~Sala\v{c}.
\newblock {Resolution of the $k$-Dirac operator}, 2017.
\newblock arXiv:1705.10168 [math.DG].

\end{thebibliography}

\end{document}